\documentclass[12pt,reqno]{amsart}
\usepackage[foot]{amsaddr}

\newcommand{\showdateandconfidential}[2]{#2}

\usepackage{graphicx}
\usepackage{fullpage}
\usepackage{etoolbox}
\usepackage[utf8]{inputenc}
\usepackage{amsmath}  
\usepackage{amssymb}     
\usepackage{amsthm}  
\usepackage{bbm}    
\usepackage{bm}
\usepackage{accents}
\usepackage{mathrsfs}
\usepackage{mathtools}
\usepackage{fixmath}
\usepackage{fullpage}
\usepackage{tikz}
\usepackage{microtype}
\usepackage{complexity}
\usepackage{upgreek,esint} 
\usepackage{xcolor}
\usepackage{stackengine}
\usepackage{subcaption}
\usepackage{stmaryrd}
\usepackage{fancyhdr}
\usepackage{subcaption}
\usepackage{comment}

\usepackage{fancyhdr}
\showdateandconfidential{
\fancyhead[CE,CO]{\footnotesize\textbf{Working paper (confidential)}}
\fancyhead[RE,RO]{\footnotesize\today}
}{}
\usepackage{amsfonts} 
\usepackage{enumerate}
\usepackage[boxed,noline,linesnumbered]{algorithm2e}
\let\oldnl\nl
\newcommand{\nonl}{\renewcommand{\nl}{\let\nl\oldnl}}
\SetInd{0.5em}{0.5em} 
\usepackage[normalem]{ulem}

\newcommand{\ubar}[1]{\text{\b{$#1$}}}

\newcommand{\hideiftight}[1]{}%

\makeatletter
\newcommand{\pushright}[1]{\ifmeasuring@#1\else\omit\hfill$\displaystyle#1$\fi\ignorespaces}
\newcommand{\pushleft}[1]{\ifmeasuring@#1\else\omit$\displaystyle#1$\hfill\fi\ignorespaces}
\makeatother

\makeatletter
\def\paragraph{\@startsection{paragraph}{4}%
  \z@\z@{-\fontdimen2\font}%
  {\normalfont\bfseries}}
\makeatother

\usepackage{hyperref} 
\hypersetup{
colorlinks=true,
linktocpage=true,
pdfstartview=FitH,
breaklinks=true,
pdfpagemode=UseNone,
pageanchor=true,
pdfpagemode=UseOutlines,
plainpages=false,
bookmarksnumbered,
bookmarksopen=false,
bookmarksopenlevel=1,
hypertexnames=true,
pdfhighlight=/O,
urlcolor=black,
linkcolor=black,
citecolor=black,
pdftitle={},
pdfauthor={},
pdfsubject={},
pdfkeywords={},
pdfcreator={pdfLaTeX},
pdfproducer={LaTeX with hyperref}
}

\usepackage[capitalise]{cleveref}

\colorlet{pricecolor}{white}
\colorlet{diffpricecolor}{white}
\colorlet{xvarcolor}{white}
\colorlet{ecolor}{white}
\colorlet{scolor}{white}
\colorlet{presentcolor}{cyan!70!blue}
\colorlet{pastcolor}{yellow!10!lime}

\usetikzlibrary{shapes}

\newtheorem{theorem}{Theorem}[section]

\newtheorem{lemma}[theorem]{Lemma}
\newtheorem{assumption}[theorem]{Assumption}

\theoremstyle{remark}

\renewcommand{\geq}{\geqslant}
\renewcommand{\leq}{\leqslant}
\renewcommand{\preceq}{\preccurlyeq}
\renewcommand{\succeq}{\succcurlyeq}

\usepackage[colorinlistoftodos,bordercolor=orange,backgroundcolor=orange!20,linecolor=orange,textsize=scriptsize]{todonotes} 
\usepackage{acronym}		
\newcommand{\acdef}[1]{\textit{\acl{#1}} \textup{(\acs{#1})}\acused{#1}}		

\newacro{OT}{optimal transport}
\newacro{DP}{dynamic programming}
\newacro{PTAS}{polynomial-time approximation scheme}
\newacro{FPTAS}{fully polynomial-time approximation scheme}
\newacro{ATP}{available to promise}
\newacro{CTP}{capable to promise}
\newacro{PTP}{profitable to promise}

\definecolor{tealblue}{HTML}{007050}
\definecolor{MyBlue}{HTML}{000070}
\definecolor{MyGreen}{HTML}{007070}
\definecolor{changes}{HTML}{0050B0}

\newcommand{\revise}[1]{{\color{changes}#1}}		
\newcommand{\reviseagain}[1]{{\color{tealblue}#1}}		
\newcommand{\try}[1]{{\color{magenta}#1}}		
\newcommand{\fixed}[1]{{\color{MyGreen}#1}}		
\newcommand{\old}[1]{}
\definecolor{Cloud}{HTML}{F5F5F5}
\newcommand{\OB}[1]{\todo[color=Cloud,bordercolor=black,author=\textbf{OB},inline]{\small #1\\}}
\newcommand{\OBoff}[1]{}
\definecolor{MScolor}{HTML}{BBDDFF}

\definecolor{NOcolor}{HTML}{FFFFF0}

\definecolor{CWcolor}{HTML}{FBF4FF}

\definecolor{EDFcolor}{HTML}{FCCBB8}

\makeatletter

\newcommand{\routinenl}{%
\nonl
  \refstepcounter{AlgoLine}%
  \nlset{\arabic{AlgoLine}}
  }
\makeatother

\newcounter{condition}
\newenvironment{condition}{\refstepcounter{condition}\equation}{\tag{C\thecondition}\endequation}

\usepackage[scaled]{inconsolata}

\usepackage{newfloat}
\usepackage{caption}
\usepackage{subcaption}

\DeclareFloatingEnvironment[fileext=frm,placement={!ht},name=Algorithm,within=none]{algorithmenv}  
\makeatletter
\renewcommand{\@algocf@capt@boxed}{above}
\makeatother
\makeatletter
\newcommand{\removelatexerror}{\let\@latex@error\@gobble}
\makeatother

\usetikzlibrary{arrows, positioning}
\usepackage{pbox}
\usepackage{mathtools}

\newcommand{\dom}{\mathrm{dom}}

\newcommand{\obsolete}[1]{}
\newcommand{\nocolsep}{\arraycolsep=1.4pt\def\arraystretch{1}}
\makeatletter
\newcommand{\refereq}[2]{\overset{\text{\tiny{#1}}}{#2}}\newcommand{\vast}{\bBigg@{4}}
\newcommand{\Vast}{\bBigg@{5}}
\makeatother
\newcommand{\Real}{\mathbb{R}}

\newcommand{\Realpluszero}{\Real_{\geq 0}}

\newcommand{\Natural}{\mathbb{N}}

\newcommand{\grad}{\nabla}
\newcommand{\gradi}[1]{\grad_{#1}}

\newcommand{\Hessian}{\nabla^2}

\newcommand{\inner}[2]{\langle{#1},{#2}\rangle}

\DeclareMathOperator{\proxoperatorfunction}{prox}
\newcommand{\proxoperator}[1]{\proxoperatorfunction_{#1}}

\newcommand{\Prox}[2]{\proxoperator{#1}\left(#2\right)}

\newcommand{\expectationfunction}{\mathbb{E}}
\newcommand{\expectation}[1]{\expectationfunction[#1]}
\newcommand{\Expectation}[1]{\expectationfunction\left[#1\right]}
\newcommand{\condexpectation}[2]{\expectationfunction[#2\cond{#1}]}
\newcommand{\Condexpectation}[2]{\expectationfunction\left[#2\cond{#1}\right]}

\newcommand{\probafunction}{\mathbb{P}}
\newcommand{\proba}[1]{\probafunction(#1)}

\newcommand{\samplespace}{\Omega}
\newcommand{\filtrationfunction}{\mathcal{F}}
\newcommand{\filtrationk}[1]{\filtrationfunction_{#1}}

\usepackage{scalerel}
\newlength\bshft 
\def\fakebold#1{\ThisStyle{\ooalign{$\SavedStyle#1$\cr%
  \kern-\bshft$\SavedStyle#1$\cr%
  \kern\bshft$\SavedStyle#1$}}}

\newcommand{\zero}[1]{o(#1)}

\newcommand{\magnitude}[1]{O(#1)}
\newcommand{\Magnitude}[1]{O\left(#1\right)}

\newcommand{\cond}{|}

\newcommand{\indicatorfunction}[1]{I_{#1}}
\newcommand{\indicator}[2]{\indicatorfunction{#1}(#2)}

\newcommand{\sigmaalgebra}{\upsigma}
\newcommand{\sigmaalgebrax}[1]{\sigmaalgebra(#1)}

\newcommand{\distfunction}{\textup{dist}}
\newcommand{\disttfunction}[1]{\distfunction_{#1}}

\newcommand{\bigdistt}[3]{\disttfunction{#1}\big(#2,#3\big)}
\DeclareMathOperator{\diag}{diag}

\DeclareMathOperator{\subjectto}{subject\ to}

\DeclarePairedDelimiter{\norm}{\|}{\|}

\newcommand{\normx}[2]{\norm{#1}_{#2}}

\newcommand{\transpose}{\top}

\newcommand{\defeq}{:=}

\newcommand{\smoothfunction}{f}
\newcommand{\smooth}[1]{\smoothfunction({#1})}

\newcommand{\nonsmoothfunction}{g}
\newcommand{\nonsmooth}[1]{\nonsmoothfunction({#1})}

\newcommand{\identity}{\textup{Id}}

\newcommand{\modulus}[1]{|{#1}|}

\newcommand{\scaledproxoperator}[2]{\proxoperatorfunction^{#1}_{#2}}
\newcommand{\scaledprox}[3]{\scaledproxoperator{#1}{#2}(#3)}
\newcommand{\scaledProx}[3]{\scaledproxoperator{#1}{#2}\left(#3\right)}

\newcommand{\iter}{k} 
\newcommand{\altiter}{l} 
 
\newcommand{\dimension}{d}
\newcommand{\primaldimension}{m}
\newcommand{\primaldimensioni}[1]{\primaldimension_{#1}}
\newcommand{\Dimension}{P}
\newcommand{\dualdimension}{q}

\newcommand{\DimensionI}[1]{\Dimension_{#1}}
\renewcommand{\probafunction}{\textup{Prob}}

\newcommand{\probas}{\pi}
\newcommand{\probai}[1]{\probas_{#1}}
\newcommand{\probaij}[2]{\probas_{#1,#2}}

\newcommand{\stepsize}{\tau}
\newcommand{\stepsizek}[1]{\stepsize^{#1}}

\newcommand{\Stepsize}{T}

\newcommand{\Stepsizek}[1]{\Stepsize}
\newcommand{\Stepsizeki}[2]{\Stepsizek{#1}_{#2}}
\newcommand{\StepsizekI}[2]{\Stepsizek{#1}_{#2}}

\newcommand{\StepsizeI}[1]{\Stepsize_{#1}}
\newcommand{\stepsizei}[1]{\stepsize_{#1}}

\newcommand{\difference}{\delta}
\newcommand{\differencemax}{\bar{\difference}}

\newcommand{\differencemin}{\ubar{\difference}}

\newcommand{\dualstepsize}{\sigma}
\newcommand{\dualstepsizek}[1]{\dualstepsize^{#1}}
\newcommand{\Diagonal}{Q}

\newcommand{\diagonal}{q}
\newcommand{\diagonali}[1]{\diagonal_{#1}}

\newcommand{\constraintcovariance}{\Sigma}
\renewcommand{\constraintcovariance}{\Xi}

\newcommand{\constraintcovarianceIJ}[2]{\constraintcovariance_{#1#2}}

\newcommand{\convexratio}{\vartheta}
\newcommand{\nomoreconvexratio}{}
\newcommand{\Convex}{\Upsilon}
\newcommand{\Convexi}[1]{\Convex_{#1}}
\newcommand{\ConvexI}[1]{\Convex_{#1}}
\newcommand{\Smoothness}{\Lambda}

\newcommand{\Smoothnessi}[1]{\Smoothness_{#1}}
\newcommand{\SmoothnessI}[1]{\Smoothness_{#1}}

\newcommand{\Scale}{\Delta}
\newcommand{\Scalek}[1]{\Scale}
\newcommand{\coefficient}{c}%
\newcommand{\coefficientk}[1]{\coefficient_{#1}}%
\newcommand{\penalized}{F}%
\newcommand{\penalizedk}[1]{\penalized_{#1}}%
\newcommand{\scalematrix}{\Psi}%
\newcommand{\scalematrixk}[1]{\scalematrix^{#1}}%

\newcommand{\NSmoothness}{\widehat{\Smoothness}}

\renewcommand{\NSmoothness}{\Smoothness}

\newcommand{\gammacoef}{\Gamma}
\newcommand{\gammacoefkl}[2]{\gammacoef_{#1}^{#2}}

  \newcommand{\stronglyconvexfunction}{\zeta}
  \newcommand{\stronglyconvex}[1]{\stronglyconvexfunction(#1)}

\newcommand{\lyapunovfunction}{\mathcal{L}}%
\newcommand{\lyapunovk}[1]{\lyapunovfunction_{#1}}%

\newcommand{\Block}{\mathcal{I}}
\newcommand{\block}{J}
\newcommand{\coordinate}{i}
\newcommand{\altcoordinate}{j}

\newcommand{\blockk}[1]{\block^{#1}}
\newcommand{\Aall}{A}

\newcommand{\AI}[1]{\Aall_{#1}}
\newcommand{\bvector}{b}

\newcommand{\linearizedlagrangiankfunction}[1]{\xi^{#1}}

\newcommand{\xall}{\bm{x}}
\newcommand{\hatxall}{\hat{\xall}}

\newcommand{\xallsol}{\xall^\ast}

\newcommand{\altxall}{\tilde{\xall}}

\newcommand{\altaltxall}{\bar{\xall}}
\newcommand{\xallk}[1]{\xall^{#1}}

\newcommand{\hatxallk}[1]{\hatxall^{#1}}

\newcommand{\xI}[1]{\xall_{#1}}
\newcommand{\xIk}[2]{\xI{#1}^{#2}}
\newcommand{\xIsol}[1]{\xallsol_{#1}}

\newcommand{\hatxI}[1]{\hatxall_{#1}}
\newcommand{\hatxIk}[2]{\hatxI{#1}^{#2}}

\newcommand{\altxI}[1]{\altxall_{#1}}
\newcommand{\Xall}{\mathcal{X}}
\newcommand{\Xallsol}{\Xall^\ast}

\newcommand{\Zall}{\bm{z}}
\newcommand{\Zallk}[1]{\Zall^{#1}}

\newcommand{\ZI}[1]{\Zall_{#1}}
\newcommand{\ZIk}[2]{\ZI{#1}^{#2}}
\newcommand{\Sall}{\bm{s}}
\newcommand{\Sallk}[1]{\Sall^{#1}}

\newcommand{\thetak}[1]{\theta_{#1}}
\newcommand{\varthetak}[1]{\vartheta_{#1}}
\newcommand{\sumdualstepsize}{\Sigma}
\newcommand{\Sigmak}[1]{\sumdualstepsize_{#1}}

\newcommand{\multiplier}{\altaltdualvariable}
\newcommand{\dualvariable}{y}
\newcommand{\altdualvariable}{u}
\newcommand{\altaltdualvariable}{v}

\newcommand{\multipliersol}{\multiplier^\ast}

\newcommand{\dualk}[1]{\dualvariable^{#1}}
\newcommand{\altdualk}[1]{\altdualvariable^{#1}}

\newcommand{\constraintfunction}{h}
\newcommand{\constraintmin}{\constraintfunction^\ast}
\newcommand{\constraint}[1]{\constraintfunction(#1)}
\newcommand{\Constraint}[1]{\constraintfunction\left(#1\right)}

\newcommand{\gradI}[1]{\grad_{#1}}
 \newcommand{\Uniti}[1]{U_{#1}}
 
 \newcommand{\UnitI}[1]{U({#1})}
 
 \renewcommand{\identity}{\bm{I}}

 \newcommand{\identityI}[1]{\identity_{#1}}

\newcommand{\spectralradiussymbol}{\upsigma_{\max}}
\newcommand{\spectralradius}[1]{\spectralradiussymbol(#1)}
\newcommand{\Spectralradius}[1]{\spectralradiussymbol\left({#1}\right)}
\newcommand{\invspectralradiussymbol}{\spectralradiussymbol^{-1}}
\newcommand{\invspectralradius}[1]{\invspectralradiussymbol(#1)}
\newcommand{\invSpectralradius}[1]{\invspectralradiussymbol\left({#1}\right)}
\newcommand{\smallesteigenvaluesymbol}{\uplambda_{\min}}
\newcommand{\smallesteigenvalue}[1]{\smallesteigenvaluesymbol(#1)}
\newcommand{\bigsmallesteigenvalue}[1]{\smallesteigenvaluesymbol\big(#1\big)}
\newcommand{\Smallesteigenvalue}[1]{\smallesteigenvaluesymbol\left({#1}\right)}

\newcommand{\lagrangianfunction}{L}
\newcommand{\lagrangian}[2]{\lagrangianfunction(#1,#2)}

\newcommand{\Mexpr}{M}
\newcommand{\Mexprx}[1]{\Mexpr(#1)}
\newcommand{\Hexpr}{H}
\newcommand{\invHexpr}{\Hexpr^{-1}}
\newcommand{\Hexprx}[1]{\Hexpr(#1)}
\newcommand{\Qexpr}{Q}
\newcommand{\Qexprx}[1]{\Qexpr(#1)}
\newcommand{\invHexprx}[1]{\invHexpr(#1)}
\newcommand{\betaparameter}{\beta}
\newcommand{\betacr}[1]{\beta(#1)}
\newcommand{\cparameter}{\gamma}
\newcommand{\conditioning}{\kappa}

\renewcommand{\fixed}[1]{#1}		

\renewcommand{\dimension}{p}
\renewcommand{\primaldimension}{n}
\renewcommand{\dualdimension}{m}
\renewcommand{\Diagonal}{Q}
\renewcommand{\diagonal}{q}

\renewcommand{\xall}{x}
\renewcommand{\Zall}{z}
\newcommand{\vcompositefunction}{\mathcal{\compositefunction}}
\newcommand{\zall}{v}
\renewcommand{\Sall}{s}

\renewcommand{\Sigmak}[1]{\sumdualstepsize^{#1}}
\renewcommand{\gammacoefkl}[2]{\gammacoef^{#1,#2}}

\newcommand{\onesvectorn}[1]{\bm{1}_{#1}}

\newcommand{\sample}{\omega}

\providecommand{\samplespace}{\upOmega}
\renewcommand{\samplespace}{\upOmega}
\newcommand{\eventspace}{\mathcal{F}}
\newcommand{\probabilitymeasure}{\mathbb{P}}
\newcommand{\probability}[1]{\probabilitymeasure({#1})}

\newcommand{\convergesinprobability}{\overset{p}{\to}}

\renewcommand{\filtrationk}[1]{\filtrationfunction^{#1}}

\providecommand{\compositefunction}{G}
\providecommand{\smoothfunction}{h}
\providecommand{\nonsmoothfunction}{g}
\providecommand{\constraintfunction}{f}

\renewcommand{\compositefunction}{G}
\renewcommand{\smoothfunction}{h}
\renewcommand{\nonsmoothfunction}{g}
\renewcommand{\constraintfunction}{f}

\newcommand{\smoothIfunction}[1]{\smoothfunction_{#1}}

\providecommand{\smooth}[1]{\smoothfunction(#1)}

\newcommand{\smoothI}[2]{\smoothIfunction{#1}(#2)}

\newcommand{\NnonsmoothIfunction}[1]{\Nnonsmoothfunction_{#1}}
\newcommand{\Nnonsmooth}[1]{\Nnonsmoothfunction(#1)}
\newcommand{\NnonsmoothI}[2]{\NnonsmoothIfunction{#1}(#2)}

\newcommand{\compositeopt}{\compositefunction^\ast}
\newcommand{\compositemin}{\compositefunction_{\min}}
\newcommand{\compositeIfunction}[1]{\compositefunction_{#1}}
\newcommand{\compositeI}[2]{\compositeIfunction{#1}(#2)}
\providecommand{\vcompositefunction}{\bm{\mathcal{\compositefunction}}}
\newcommand{\hatcomposite}{\hat{\compositefunction}}
\newcommand{\hatvcomposite}{\hat{\vcompositefunction}}
\newcommand{\vcomposite}[1]{\vcompositefunction(#1)}
\newcommand{\hatcompositek}[1]{\hatcomposite^{#1}}
\newcommand{\hatvcompositek}[1]{\hatvcomposite^{#1}}

\newcommand{\composite}[1]{\compositefunction(#1)}

\newcommand{\Nnonsmoothfunction}{\nonsmoothfunction}
\renewcommand{\NSmoothness}{\Smoothness}

\newcommand{\nonsmoothIfunction}[1]{\nonsmoothfunction_{#1}}
\providecommand{\nonsmooth}[1]{\nonsmoothfunction(#1)}

\newcommand{\nonsmoothI}[2]{\nonsmoothIfunction{#1}(#2)}
\renewcommand{\sumdualstepsize}{S}
\renewcommand{\dualstepsize}{\sigma}

\renewcommand{\Xall}{\Real^\primaldimension}%
\renewcommand{\Xallsol}{\mathcal{S}}%
\renewcommand{\block}{B}%

\newcommand{\blockkw}[2]{\blockk{#1}({#2})}
\renewcommand{\identity}{I}
\renewcommand{\indicatorfunction}[1]{\mathbb{I}_{#1}}

\newcommand{\crapfunction}{u} 
\newcommand{\crap}[1]{\crapfunction({#1})} 
 
\renewcommand{\lyapunovfunction}{V}
\newcommand{\lyapunovkfunction}[1]{\lyapunovfunction^{#1}}
\renewcommand{\lyapunovk}[2]{\lyapunovkfunction{#1}(#2)}
\newcommand{\scaledlyapunovfunction}{W}%
\newcommand{\scaledlyapunovkfunction}[1]{\scaledlyapunovfunction^{#1}}%
\newcommand{\scaledlyapunovk}[2]{\scaledlyapunovkfunction{#1}({#2})}%

\newcommand{\constantC}{C}
\newcommand{\constantCprime}{C'}

\newcommand{\indexsetexp}{I}%
\newcommand{\sublevelsetsymbol}{\mathcal{L}}%
\newcommand{\sublevelset}[1]{\mathcal{L}({#1})}%
\newcommand{\meansublevelset}{\bar{\sublevelsetsymbol}}%
\newcommand{\minsublevelset}{\hat{\sublevelsetsymbol}}%
\newcommand{\meanSall}{\bar{\Sall}}%
\newcommand{\meanSallk}[1]{\meanSall^{#1}}%
\newcommand{\meanSallkc}[2]{\meanSallk{#1}_{#2}}%
\newcommand{\Sallkc}[2]{\hat{\Sall}^{#1}_{#2}}%
\newcommand{\meanSallklimit}{\tilde{\Sall}}%
\newcommand{\residfunction}{\varphi}%
\newcommand{\resid}[1]{\residfunction({#1})}%
\newcommand{\residmax}{\residfunction_{\max}}%

\newcommand{\cstepsilon}{\epsilon}%
\newcommand{\cstvarepsilon}{\varepsilon}%
\newcommand{\cstvareps}[1]{\cstvarepsilon^\star({#1})}%
\newcommand{\cstalpha}{\alpha}%
\newcommand{\cstalphamin}{\cstalpha_{\min}}%
\newcommand{\cstdelta}{\delta}%
\newcommand{\cstdelt}[1]{\cstdelta^\star({#1})}%
\newcommand{\cstmu}{\mu}%
\newcommand{\cstnu}{\nu}%
\newcommand{\cstt}[1]{t({#1})}%
\newcommand{\compositemax}[1]{\bar{\compositefunction}({#1})}%
\newcommand{\itermin}{\iter_{\min}}%
\newcommand{\iterhat}{\hat{\iter}}%
\newcommand{\iterKfunction}{\bar{\iter}}%
\newcommand{\iterK}[1]{\iterKfunction({#1})}%
\newcommand{\bigiterK}[1]{\iterKfunction\big({#1}\big)}%
\newcommand{\iterTfunction}{\tilde{\iter}}%
\newcommand{\iterT}[1]{\iterTfunction({#1})}%
\newcommand{\bigiterT}[1]{\iterTfunction\big({#1}\big)}%
\providecommand{\zall}{\bm{v}}%
\newcommand{\zallk}[1]{\zall^{#1}}%
\newcommand{\zalllimit}{\hat{\zall}}%
\newcommand{\betak}[1]{\beta^{#1}}%
\newcommand{\gammak}[1]{\gamma^{#1}}%

\newacro{PDA}{primal-dual algorithm}
\newacro{SOCP}{Second-Order Conic Programming}
\newacro{DSO}{Distribution System Operator}
\newacro{LA}{Load Aggregator}
\newacro{OPF}{Optimal Power Flow}
\newacro{AC}{alternative current}
\newacro{DER}{Distributed Energy Resource}
\newacro{LMP}{Locational Marginal Price}
\newacro{DLMP}{Distribution Locational Marginal Price}
\newacro{OPF}{Optimal Power Flow}
\newacro{ACOPF}{Alternative Current Optimal Power Flow}
\newacro{PPDLMP}{privacy-preserving DLMP solver}
\newacro{ESO}{expected separable overapproximation}

\newacro{l.s.c}{lower semi-continuous}

\usepackage[authoryear]{natbib}

\title{Parametrization and convergence of a primal-dual block-coordinate approach to linearly-constrained nonsmooth optimization 
}

\author{Olivier Bilenne}

\address[Olivier Bilenne]{Laboratoire Informatique d’Avignon, Avignon, France}
\email{olivier.bilenne@univ-avignon.fr}

\thanks{This research was supported by the Gaspard Monge Program for optimization, operations research and their interactions with data sciences (PGMO). 
It was done in part at the Department of Data Science and Knowledge Engineering, Maastricht University, Netherlands, and in part at CERMICS, \'Ecole des Ponts ParisTech, France. The author is now with Laboratoire Informatique d’Avignon.}

\begin{document}

\begin{abstract}
This note is concerned with the problem of minimizing a separable, convex, composite (smooth and nonsmooth)
function subject to linear constraints. We study a randomized block-coordinate interpretation of the Chambolle-Pock primal-dual algorithm, based on inexact proximal gradient steps. A specificity of the considered algorithm is its robustness, as it converges even in the absence of strong duality or when the linear program is inconsistent. Using matrix preconditiong, we derive tight sublinear convergence rates with and without duality assumptions and for both the convex and the strongly convex settings. Our developments are extensions and particularizations of original algorithms proposed by~\cite{malitsky17} and~\cite{luke18}. Numerical experiments are provided for an optimal transport problem of service pricing.
\end{abstract} 

 

\maketitle
\thispagestyle{fancy}

\OBoff{Remove bold font to $\xall,\zall,\vcompositefunction,\hatvcompositek{\iter}$.}
\OBoff{Rename $\dimension,\primaldimension,\dualdimension,\diagonal,\Diagonal$ as in reference paper.}
\OBoff{Change $\Sigmak{\iter},\gammacoefkl{\iter}{\altiter}$.}
 \OBoff{Add applications examples.}
\section{Introduction}

%
%

\label{sec:Problem}

In an $\dimension$-agent network, we consider the problem 
\begin{equation}\label{initialP}\tag{P}
\min_{\xall\in\Xall} \ 
\nonsmooth{\xall} + \smooth{\xall} 
   \quad \subjectto \quad
    \Aall\xall=\bvector 
,
\end{equation}
where $\Aall\in\Real^{\dualdimension\times \primaldimension}$, $\bvector\in\Real^{\dualdimension}$,
and $\xall= (\xI{1},\ldots,\xI{\dimension})\in\Xall$
contains~$\dimension$ decision variables $\xI{1}\in\Real^{\primaldimensioni{1}},\ldots,\xI{\dimension}\in\Real^{\primaldimensioni{\dimension}}$, each locally assigned to one of the agents ($\primaldimension=\sum_{\coordinate=1}^{\dimension}\primaldimensioni{i}$). 
The functions~$\smoothfunction$ and~$\nonsmoothfunction$ are assumed to be additively separable for the local variables, i.e.,
\[
\smooth{\xall} = \sum\limits_{\coordinate=1}^{\dimension} \smoothI{\coordinate}{\xI{\coordinate}}
,
\qquad
\nonsmooth{\xall} = \sum\limits_{\coordinate=1}^{\dimension} \nonsmoothI{\coordinate}{\xI{\coordinate}}
,
\]
where local functions $ \smoothIfunction{1},\dots,\smoothIfunction{\dimension}$ are smooth and differentiable, and $ \nonsmoothIfunction{1},\dots,\nonsmoothIfunction{\dimension}$ are convex lower semi-continuous (possibly non-smooth) functions with easily computable resolvents. We call $ \composite{\xall} = \sum_{\coordinate=1}^{\dimension} \compositeI{\coordinate}{\xI{\coordinate}} $ the total objective function, where $\compositeI{i}{\xI{i}}= \nonsmoothI{i}{\xI{i}} + \smoothI{i}{\xI{i}}$, and we give~$\Aall$ the block matrix structure $\Aall = (\AI{1} \, \AI{2} \,\cdots\, \AI{\dimension})$, with $\AI{1}\in\Real^{\dualdimension\times \primaldimensioni{1}},\dots,\AI{\dimension}\in\Real^{\dualdimension\times \primaldimensioni{\dimension}}$, so that the equality constraint in~\eqref{initialP} rewrites as $\sum_{\coordinate=1}^{\dimension}\AI{i}\xI{\coordinate}=\bvector$.
Problem~\eqref{initialP} finds application in numerous fields, including linear programming, optimal transport, composite minimization, distributed optimization, and inverse problems.

To solve~\eqref{initialP} we consider a semi-distributed algorithm in the spirit of the \ac{PDA} of~\cite{chambolle11}. 
Our developments build on the convergence study issued in~\cite{malitsky17} for the primal-dual method as applied to equality-constrained convex optimization. In the original analysis, the problem is reformulated as an instance of the more general constrained optimization problem
\begin{equation}\label{genericP}\tag{P'} 
   \min_{\xall\in\Xall}\
   \smooth{\xall} + \nonsmooth{\xall}  
\quad
  \subjectto 
  \quad
  \xall\in\arg\min_{\altxall}\constraint{\altxall}
  ,
\end{equation}
where the residual function~$\constraintfunction$ is given the quadratic form 
\begin{equation}\label{constraint}
\constraint{\xall} = \frac 1 2 \norm{\Aall\xall-\bvector}^{2}
.
\end{equation}
The primal-dual algorithm is then regarded as a variant of the accelerated proximal gradient method of~\cite{tseng08apgm}.
An advantage of this viewpoint, as compared the classic primal-dual analysis, is that convergence can be guaranteed in the absence of a constraint qualification for strong duality, or even when the constraint $ \Aall\xall=\bvector  $ is inconsistent and~\eqref{initialP} admits no solution, in which case~\eqref{genericP} remains feasible.
We refer to~\cite{luke18} for a thorough discussion on the benefits of this particular approach to solving~\eqref{initialP}.

In the present work we revisit the coordinate-descent implementation of the \ac{PDA} proposed in~\cite{luke18}, extending the latter to random coordinate block selection, and to composite objective functions with smooth and nonsmooth parts by implementing inexact proximal gradient steps. The demand for proximal gradient methods, in particular, is high in the practical appplications where proximal steps cannot be implemented exactly.

For the case of strongly convex objectives, we derive a specific stepsize sequence that leads to accelerated convergence rates, which we obtain using suitable matrix preconditioners. To our knowledge, these faster rates were, to date, only available for the centralized implementation of the algorithm.
Numerical experiments are   eventually reported for a simple problem in optimal transport.


\subsection*{Notation}
In the real space~$\Real^\primaldimension$, the identity matrix  is denoted by $\identityI{\primaldimension}$, the vector of ones by $\onesvectorn{\primaldimension}=(1,1,\dots,1)$, and~$\diag(a_1,\dots,a_\primaldimension)$ is the diagonal matrix with~$a_1,\dots,a_\primaldimension$ as diagonal entries. 
For $\xall,\Zall\in\Real^{\primaldimension}$, we let $\inner{\xall}{\Zall}=\sum_{\coordinate=1}^{\primaldimension}\xI{\coordinate}\ZI{\coordinate}=\xall^\transpose\Zall$ denote the Euclidean inner product of~$\xall$ and~$\Zall$, and $\norm{\xall}=\inner{\xall}{\xall}^{1/2}$ the Euclidean norm of~$\xall$. Given a symmetric and positive definite matrix $\Lambda\in\Real^{p\times p}$, we write $\normx{\xall}{\Lambda}=\inner{\Lambda\xall}{\xall}^{1/2}$ for $\xall\in\Real^{\primaldimension}$.  
We call~$\mathcal{S}$ the set of solutions of \eqref{genericP} and call~$\compositeopt$ its optimal value, so that $\composite{\xallsol}=\compositeopt$ for all~$\xallsol\in\mathcal{S}$.

\section{Algorithm and main results}
\label{sec:Algo}
\renewcommand{\Stepsizek}[1]{\Stepsize^{#1}}%
\renewcommand{\Scalek}[1]{\Scale^{#1}}%
\renewcommand{\linearizedlagrangiankfunction}[1]{\varphi^{#1}}

Before discussing the algorithm, we formulate the assumptions that are made in the introduction on the convexity of~$\nonsmoothfunction$ and the smoothness of~$\smoothfunction$.
\begin{assumption} \label{assumption:smoothnonsmooth}
For $\coordinate=1,\ldots,\dimension$:
\begin{enumerate}[(i)]
 \item \label{assumption:smoothnonsmooth:nonsmooth}
The effective domain $\dom(\nonsmoothIfunction{\coordinate})\subseteq\Real^{\primaldimensioni{i}}$ is nonempty, closed, and convex. There exists a  symmetric, positive semi-definite matrix $\Convexi{\coordinate}\in\Real^{\primaldimensioni{\coordinate}\times \primaldimensioni{\coordinate}}$ such that 
$\NnonsmoothI{\coordinate}{\cdot}-\frac{1}{2}\normx{\cdot}{\Convexi{\coordinate}}^2$ 
is convex. 
\item \label{assumption:smoothnonsmooth:smooth}
Function~$\smoothIfunction{\coordinate}$ is convex differentiable, and there exists a symmetric, positive semi-definite matrix $\Smoothnessi{\coordinate}\in\Real^{\primaldimensioni{\coordinate}\times \primaldimensioni{\coordinate}}$ such that, for every $\xall_{\coordinate},\tilde{\xall}_{\coordinate}\in\dom(r_{\coordinate})$,
\begin{align}
 \label{smoothnessconvexity}
0
 \leq  \smoothI{\coordinate}{\altxI{\coordinate}} - \smoothI{\coordinate}{\xall_{\coordinate}} -   \inner{\grad\smoothI{\coordinate}{\xall_{\coordinate}}}{\altxI{\coordinate}-\xI{\coordinate}} \leq  \frac{1}{2}\normx{\altxI{\coordinate}-\xI{\coordinate}}{\Smoothnessi{\coordinate}}^2
.
\end{align}
\end{enumerate}
We write $\Convex=\diag(\ConvexI{1},\dots,\Convexi{\dimension})$, and $\Smoothness=\diag(\SmoothnessI{1},\dots,\Smoothnessi{\dimension})$.
\end{assumption}\noindent
\OBoff{Since the stepsizes are chosen to satisfy a global criterion, check if convexity cannot be defined by full blocks~$\block$, by recomputing  a (now non-diagonal) matrix~$\Convex$ in the developments. If okay, then discuss the shape of~$\Convex$, the choice of the stepsize, and the look of the algorithm (coordinate descent or block-coordinate) before going any further. Simultaneously, introduce condition on  stepsize for convergence).   }%
We note that in the general convex problem the matrices~$\Convex$ and~$\Smoothness$ are positive semi-definite ($\Convex,\Smoothness\succeq 0$). 
Strong convexity of~$\compositefunction$ implies as well the existence of a positive definite~$\Convex\succ 0$.
Under these assumptions, we consider the following iterative algorithm (\cref{algorithm:newgenericprimaldual}), in which we use
\[
\scaledprox{M}{f}{\Zall}
=
(I+M^{-1}\partial{f})^{-1}\Zall
=
\arg\min\nolimits_{\xall\in\Real^\primaldimension}\left\{f(\xall)+\frac 1 2 \normx{\xall-\Zall}{M}^2\right\}
\]
%
%
%
%
%
%
to denote the (scaled) proximal operator associated with a function~$f$ and a symmetric and positive definite scaling matrix~$M$.
\begin{algorithm}
\caption{Primal-dual block coordinate descent
\label{algorithm:newgenericprimaldual}}
\DontPrintSemicolon
\SetAlgoNoLine%
\SetKwInOut{Parameters}{{\textbf{Parameters}}}
\SetKwInOut{Init}{{\textbf{Initialization}}}
\SetKwInOut{Output}{{\textbf{Output}}}
\SetKwFor{For}{for}{do}{} 
\Parameters{ 
$\Dimension$, $ \Stepsizek{\iter} $, $\dualstepsizek{\iter}$
} 
\Init{$\xallk{0}\old{=\Sallk{0}}\in\Xall$, $\dualk{0} =   \dualstepsizek{0} (\Aall\xallk{0}-\bvector)$, $\altdualk{0} =  \Aall\xallk{0}-\bvector$} 
\Output{$\xallk{\iter}$
}
\smallskip 
\SetAlgoLined
 \nonl \For{$\iter=0,1,2,\dots$}{
\old{
\routinenl
$\Zallk{\iter} = (1-\thetak{\iter}) \Sallk{\iter} + \thetak{\iter} \xallk{\iter} $ \; \label{algorithm:genericPDZallk}
}
\hideiftight{
$ \varthetak{\iter+1} = \thetak{\iter}  (1-\thetak{\iter+1}) /\thetak{\iter+1} $ \; \label{algorithm:genericvarthetak}
}%
\routinenl
\text{{\textbf{select random block $\blockkw{\iter}{\sample}\subset\{1,\dots,\dimension\}$} }} \label{algorithm:primaldual:selection} \;
 \nonl \For{$\coordinate=1,\dots,\dimension$}{
   \routinenl \lIf{$\coordinate\in\blockkw{\iter}{\sample}$}{
     $ \xIk{\coordinate}{\iter+1} 
     = 
     \scaledProx{\DimensionI{\coordinate} \StepsizekI{k}{\coordinate}}{\NnonsmoothIfunction{\coordinate}}{\xIk{\coordinate}{\iter}-(\DimensionI{\coordinate} \StepsizekI{k}{\coordinate})^{-1}(\grad\smoothI{\coordinate}{\xIk{\coordinate}{\iter}}+\AI{\coordinate}^\transpose\dualk{\iter})}
     $ 
   } \label{algorithm:primaldual:genericxIk}
  \routinenl \lElse{
     $ \xIk{\coordinate}{\iter+1} = \xIk{\coordinate}{\iter} $
   }
 }
%
%
\old{
\routinenl
$\Sallk{\iter+1}  =   \Zallk{\iter}    +   \thetak{\iter} \Dimension (\xallk{\iter+1} -\xallk{\iter}) $ \;\label{algorithm:genericPDSallk}
}
\routinenl
$ \altdualk{\iter+1} 
 =  \altdualk{\iter} + \Aall ( \xallk{\iter+1}  -  \xallk{\iter} )
$ \; \label{algorithm:genericaltdualk}
\routinenl
$  \dualk{\iter+1} =  \dualk{\iter} + \dualstepsizek{\iter}     \Aall \Dimension (\xallk{\iter+1} -\xallk{\iter})   + \dualstepsizek{\iter+1} \altdualk{\iter+1}
$ \; \label{algorithm:genericdualk}
}
\end{algorithm}%

\noindent
At each step, \cref{algorithm:newgenericprimaldual}  proceeds on \cref{algorithm:primaldual:genericxIk} to individual proximal gradient steps for the primal vectors $\xIk{1}{\iter},\dots,\xIk{\dimension}{\iter}$ along a randomly selected subset of coordinate direction blocks. 
The full dual vector~$\dualk{\iter}$ is then upated on \cref{algorithm:genericdualk},
using the dual residual $\altdualk{\iter+1} =  \Aall\xallk{\iter+1}-\bvector$ (recursively computed on \cref{algorithm:genericaltdualk}) and a linear extrapolation term derived from the last two primal iterates. 
The random block selection routine on \cref{algorithm:primaldual:selection} consists in drawing successive subsets of coordinate block indices; this is done by simulating a Bernoulli process~$\block:\samplespace\mapsto (2^{\{1,\dots,\dimension\}})^{\Natural}$ with underlying probability space $(\samplespace,\eventspace,\probabilitymeasure)$, where the probability measure~$\probabilitymeasure$ satisfies $\probability{\blockk{k}=\emptyset}= 0$ and $\probability{\coordinate \in \blockk{k} \cond  \blockk{-k}  } = \probability{\coordinate \in \blockk{k} } = \probai{\coordinate} > 0$ at each step~$k$ and for any block index $\coordinate\in\{1,\dots,\dimension\}$.
In the algorithm, a weighting matrix 
$\Dimension =\diag(\DimensionI{1},\dots,\DimensionI{\dimension}) \succ 0 $ is combined with a sequence of block diagonal scaling matrices $\Stepsizek{\iter} =\diag(\Stepsizeki{\iter}{1},\dots,\Stepsizeki{\iter}{\dimension})$ where, for $\coordinate=1,\dots,\dimension$, $\DimensionI{\coordinate}\defeq({1}/{\probai{\coordinate}})\,\identityI{\primaldimensioni{\coordinate}}$ and  $ \Stepsizeki{\iter}{\coordinate}\in\Real^{\primaldimensioni{\coordinate}\times \primaldimensioni{\coordinate}}$ is positive definite.
The products~$ \DimensionI{\coordinate} \StepsizekI{k}{\coordinate}$ are then used as a sequence of preconditioners for the proximal gradient steps.

\newcommand{\scalingsymbol}{K}%
\newcommand{\scalingIk}[2]{\scalingsymbol_{#1}^{#2}}%

It can be seen that the global implementation of \cref{algorithm:newgenericprimaldual} obtained by setting~$\dimension=1$ with~$\Dimension=\identityI{\primaldimension}$ and~$\dualstepsize\equiv\dualstepsize$ yields the dual update rule $  \dualk{\iter+1} 
=  
\dualk{\iter} + \dualstepsize     [\Aall  (2\xallk{\iter+1} -\xallk{\iter})   -\bvector]
$, in which case the algorithm then reduces to a simple application the \ac{PDA} with parameter~$\theta=1$.  The main differences between \cref{algorithm:newgenericprimaldual} and the method studied in \cite{luke18} are (i) the proximal steps are not computed exactly, (ii) the primal vectors are updated simultaneously in random subsets, and (iii) there is flexibility in the choice of the step size sequences~$\Stepsizek{\iter}$ and~$\dualstepsizek{\iter}$.
Also, we consider two distinct implementations of~\cref{algorithm:newgenericprimaldual}, with either constant or decreasing stepsizes.

 \subsection*{Constant stepsize.}
 With fixed parameters $\Stepsizek{\iter}\equiv\Stepsize$ and $\dualstepsizek{\iter}\equiv\dualstepsize$, Algorithm~\ref{algorithm:newgenericprimaldual} is a mere extension of~\cite{luke18} to composite objective functions with updates involving random blocks of coordinates.
A suitable choice for the preconditioner~$\Stepsize$, then, is to set
\begin{subequations}\label{earlydiagonalscaling}
\begin{equation}\label{earlydiagonalscalingpartone}
 \StepsizeI{\coordinate} =  \frac{\identityI{\primaldimensioni{\coordinate}}}{\stepsizei{\coordinate}} +\probai{\coordinate} \SmoothnessI{\coordinate}  
  +\dualstepsize\AI{\coordinate}^\transpose\AI{\coordinate}
  \qquad 
  \coordinate=1,\dots,\dimension 
 \end{equation}
%
for some nonegative parameters~$(\stepsizei{1},\dots,\stepsizei{\dimension})$ and~$\dualstepsize$
satisfying 
\begin{equation}\label{stepsizeconstraint}
  \diag\left( \frac{1}{\probai{1}}\Big(\frac{\identityI{\primaldimensioni{1}}}{\stepsizei{1}}+\dualstepsize\AI{1}^\transpose\AI{1}\Big),\dots,\frac{1}{\probai{\dimension}}\Big(\frac{\identityI{\primaldimensioni{\dimension}}}{\stepsizei{\dimension}}+\dualstepsize\AI{\dimension}^\transpose\AI{\dimension}\Big)\right) 
  -
  \dualstepsize \constraintcovariance
  \succ 
  0,
  \end{equation}
\end{subequations}
where 
we define 
$
\constraintcovariance
= 
(\constraintcovarianceIJ{\coordinate}{\altcoordinate})
$, with $\constraintcovarianceIJ{\coordinate}{\altcoordinate}={\probaij{\coordinate}{\altcoordinate}}\AI{\coordinate}^\transpose\AI{\altcoordinate}/({\probai{\coordinate}\probai{\altcoordinate}})$ and $\probaij{\coordinate}{\altcoordinate}=
\probability{\{\sample\in\samplespace \cond \coordinate,\altcoordinate \in \sample \}}
$.
%
%
\OBoff{The matrix~$\constraintcovarianceIJ{\coordinate}{\altcoordinate}$ only has nondiagonal terms when the the coordinates are updated by blocks ($\probaij{\coordinate}{\altcoordinate}\neq 0$ if $\coordinate\neq\altcoordinate$). I've been trying to compute a bound for $\spectralradius{\constraintcovariance}$ as a function of $(\probai{\coordinate})$, $(\probaij{\coordinate}{\altcoordinate})$ and $(\spectralradius{\AI{\coordinate}})$, which doesn't harm privacy too much.
}%
We note that the matrix~$\constraintcovariance$ features nondiagonal terms only when the the coordinates are updated by blocks, i.e., if $\probaij{\coordinate}{\altcoordinate}\neq 0$ for some $\coordinate\neq\altcoordinate$.

\cref{theorem:convergence} characterizes the convergence of~\cref{algorithm:newgenericprimaldual} when stepsize~\eqref{earlydiagonalscaling} is used. Its statement involves a particular sequence
$\Sallk{\iter}=\frac{1}{\iter}\sum_{\altiter=1}^{\iter}\xallk{\altiter}$ or, equivalently,
\begin{equation}\label{ergodicaverage}
\Sallk{\iter}  
=  (\identityI{\primaldimension}-\Dimension) \left( \frac{\sum_{\altiter=0}^{\iter-1} \dualstepsizek{\altiter}    \xallk{\altiter}}{ \sum_{\altiter=0}^{\iter-1} \dualstepsizek{\altiter} } \right) + \Dimension  \left( \frac{\sum_{\altiter=0}^{\iter-1}   \dualstepsizek{\altiter}  \xallk{\altiter+1}}{ \sum_{\altiter=0}^{\iter-1} \dualstepsizek{\altiter} }\right)
,
\end{equation}
which has the quality of a weighted time average of the primal iterates.
\begin{theorem}[Constant stepsize] \label{theorem:convergence}
Let sequence~$(\xallk{\iter})_{k}$ be issued by Algorithm~\ref{algorithm:newgenericprimaldual} with parameters $\Stepsizek{\iter}\equiv\Stepsize,\dualstepsizek{\iter}\equiv\dualstepsize$ satisfying~\eqref{earlydiagonalscaling}, and consider the  sequence~$(\Sallk{\iter})_{k}$ such that $\Sallk{\iter}=\frac{1}{\iter}\sum_{\altiter=1}^{\iter}\xallk{\altiter}$. 
\begin{enumerate}[(i)]
\item \label{theorem:convergence:i}
If there exists a Lagrange multiplier for Problem~\eqref{genericP}, then~$(\xallk{\iter})$ and~$(\Sallk{\iter})$ almost surely converge to a solution of~\eqref{genericP} and, almost surely, $\constraint{\xallk{\iter}}-\constraintmin = \zero{1/\iter}$, $\constraint{\Sallk{\iter}}-\constraintmin = \magnitude{1/\iter^2}$. 
\item \label{theorem:convergence:ii}
If~$\mathcal{S}$ is a bounded set and~$\smoothfunction+\nonsmoothfunction$ is bounded from below, then 
all limit points of~$(\Sallk{\iter})$ almost surely belong to~$\Xallsol$ with $\constraint{\Sallk{\iter}}-\constraintmin = \magnitude{1/\iter}$, and $\expectation{ \constraint{\Sallk{\iter}}-\constraintmin } = \zero{1/\iter}$.
\end{enumerate}
\end{theorem}

 \subsection*{Decreasing stepsizes.}
Accelerated convergence rates can be obtained when the objective function is strongly convex, i.e., if $\Convex=\diag(\ConvexI{1},\dots,\ConvexI{\dimension})\succ 0$. In the strongly convex case, we set $\Stepsizek{\iter}=\Stepsize/\stepsizek{\iter}$, and
\begin{equation}\label{newdualstepsizek}
\dualstepsizek{\iter} = \frac{\alpha}{\stepsizek{\iter}} - \betaparameter
,
\end{equation}
where 
$\alpha=\invspectralradius{\constraintcovariance\Convex^{-1}\Dimension}$ and
$
\betaparameter
=
\spectralradius{\NSmoothness\Convex^{-1}\Dimension}\, \alpha
$, in which~$\spectralradius{\cdot}$ denotes the spectral radius. Define
\begin{equation*}
 \conditioning = \frac{\betaparameter}{\alpha} = 
\spectralradius{\NSmoothness\Convex^{-1}\Dimension}
.
\end{equation*}
Given an initial  $\stepsizek{0}<1/\conditioning$, the stepsize~$\stepsizek{\iter+1}$ is recursively computed from~$\stepsizek{\iter}$ using
\begin{align} \label{stepsizesequence}
\stepsizek{\iter+1}
&\displaystyle =
\max_{\coordinate\in\{1,\dots,\dimension\}} \left\{
\frac{  \frac{1}{2} \big(\frac{1}{\probai{\coordinate}}-1-\conditioning\big) (\stepsizek{\iter})^2  + \sqrt{ \left(1 +  \frac{1}{2}\big( \frac{1}{ \probai{\coordinate}}  - \conditioning \big)\stepsizek{\iter}\right)^2 -  \frac{1}{4}    \big( \frac{2}{\probai{\coordinate}}-1  +   2\conditioning    \big)  (\stepsizek{\iter})^2}\ \stepsizek{\iter}
}{
   1 +(\frac{1}{\probai{\coordinate}}-\conditioning)\stepsizek{\iter}  
 - \conditioning (\stepsizek{\iter})^2  } 
\right\}
,
 \end{align}
The primal stepsize sequence~\eqref{stepsizesequence} is shown in Section~\ref{section:newstronglyconvexanalysis} to have the asymptotic behavior
\begin{equation*}\label{stepsizeasymptoticrate}
\stepsizek{\iter+1}
-
\stepsizek{\iter}
+
\frac{(\stepsizek{\iter})^2}{2}
=
\Magnitude{(\stepsizek{\iter})^3}
,
\end{equation*}
so that~$\stepsizek{k}$ is~$\magnitude{1/\iter}$ decreasing with exact local convergence rate $2/\iter$, and the stepsize~$\dualstepsizek{k}$ for the dual iterates increases as~$\magnitude{\iter}$.
%
%
%
Alternatively, one can use a stepsize sequence given by the  polynomial computation rule
\begin{align} \label{stepsizesequencepolynomial}
\stepsizek{\iter+1}
& =
\stepsizek{\iter}
-
\frac{(\stepsizek{\iter})^2}{2}
+  \cparameter
 \left( \frac{17}{8}\difference^2 + \frac{3}{4}\difference + \frac{1}{8} +  \conditioning  \right)
(\stepsizek{\iter})^3
 + \cparameter \left(2 \difference + \frac{1}{2}\right)   \conditioning \, (\stepsizek{\iter})^4 
 ,
 \end{align}
where $(\cparameter -1)>0$ and  $ \difference = \max_{\coordinate\in\{1,\dots,\dimension\}} \{\modulus{ \frac{1}{\probai{\coordinate}}-\conditioning}\}$,  which decreases as
\begin{equation}\label{stepsizerate}
\stepsizek{\iter+1}
\leq
 \stepsizek{\iter}
-
 \left(\frac{1-\epsilon}{2}\right)(\stepsizek{\iter})^2
\end{equation}
on condition that~$\stepsizek{0}$ is taken small enough.  An actual upper bound on the initial stepsize can be obtained by combining the conditions~\eqref{discriminantnonnegative}, \eqref{newconditionstepsizeoneexplicit} and~\eqref{newconditionstepsizetwo} derived in~\cref{section:newstronglyconvexanalysis}. 

The following accelerated convergence rates for Algorithm~\ref{algorithm:newgenericprimaldual} were derived by adjusting to our distributed setting the discussions of \cite[Section~4.2]{malitsky17}.

\OBoff{Define~$\compositeopt=\composite{\xallsol}$ for all~$\xallsol$ solution of~\eqref{genericP}.}

\begin{theorem}[{Strong convexity}] \label{theorem:newscconvergence}
Let~$\Convex\succ 0$, and denote by~$\xallsol$ the solution of~\eqref{genericP}. Let~$(\xallk{\iter})_{k}$ be a sequence issued 
by Algorithm~\ref{algorithm:newgenericprimaldual} with parameter~$\dualstepsizek{\iter}$ given by~\eqref{newdualstepsizek} and with~$\Stepsizek{\iter}=\Dimension^{-2}\Convex/\stepsizek{\iter}$, where~$\stepsizek{\iter}$ satisfies either~\eqref{stepsizesequence} or~\eqref{stepsizesequencepolynomial}. Consider the  sequence~$(\Sallk{\iter})_{k}$ such that $\Sallk{\iter}=\frac{1}{\iter}\sum_{\altiter=1}^{\iter}\xallk{\altiter}$. 
\begin{enumerate}[(i)]
\item
\label{theorem:newscconvergence:i}
If there exists a Lagrange multiplier for Problem~\eqref{genericP}, then~$(\xallk{\iter})$ almost surely converges to~$\xallsol$ with rate $\normx{\xallk{\iter}-\xallsol}{\Dimension^2\Stepsize}=\magnitude{1/\iter}$ and, almost surely, $\constraint{\xallk{\iter}}-\constraintmin=\zero{1/\iter^3}$, $\constraint{\Zallk{\iter}}-\constraintmin=\magnitude{1/\iter^4}$, $\constraint{\Sallk{\iter}}-\constraintmin=\magnitude{1/\iter^4}$.
%
\item 
\label{theorem:newscconvergence:ii}
Otherwise, 
%
%
$(\Sallk{\iter})$ almost surely converges  to~$\xallsol$ with $\constraint{\Sallk{\iter}}-\constraintmin=\magnitude{1/\iter^2}$, and $\expectation{ \constraint{\Sallk{\iter}}-\constraintmin } = \zero{1/\iter^2}$.
\end{enumerate}
\end{theorem}
%


 \section{Convergence analysis}
 \label{section:convergenceanalysis}

  \subsection{Interpretation as a proximal gradient algorithm}

  The proofs for Theorems~\ref{theorem:convergence} and~\ref{theorem:newscconvergence} 
rest on rewriting Algorithm~\ref{algorithm:newgenericprimaldual} as the (primal-only) proximal gradient algorithm given as Algorithm~\ref{algorithm:newgenericprimalonly}, which shares similarities with the accelerated proximal gradient algorithm of~\cite{tseng08apgm}, now used with stepsize $ \theta^\iter\equiv\dualstepsizek{\iter}/\Sigmak{\iter}$, where $\Sigmak{\iter}= \sum\nolimits_{\altiter=0}^{\iter}\dualstepsizek{\altiter}$.

\smallskip

\begin{algorithm}[H]
\caption{Expression of~\cref{algorithm:newgenericprimaldual} as an accelerated proximal gradient\label{algorithm:newgenericprimalonly}}
\DontPrintSemicolon
\SetAlgoNoLine%
\SetKwInOut{Parameters}{{\textbf{Parameters}}}
\SetKwInOut{Init}{{\textbf{Initialization}}}
\SetKwFor{For}{for}{do}{}
\SetKwInOut{Output}{{\textbf{Output}}}
\Parameters{$\Dimension$, $ \Stepsizek{\iter} $, $\dualstepsizek{\iter}$} 
\Init{$\xallk{0}=\Sallk{0}\in\Xall$, $\Sigmak{-1}=0$, $\Sigmak{0}=\dualstepsizek{0}$} 
\Output{$\xallk{\iter}$, $ \Sallk{\iter}  
$
}
\smallskip 
\SetAlgoLined
 \nonl \For{$\iter=0,1,2,\dots$}{
%
%
\routinenl 
$\Zallk{\iter} = \big(1/{\Sigmak{\iter}}\big)\, \big({\Sigmak{\iter-1}\Sallk{\iter} + \dualstepsizek{\iter}\xallk{\iter}}\big) $ \; \label{algorithm:genericZallk}
\routinenl
\text{{\textbf{select random block $\blockkw{\iter}{\sample}\subset\{1,\dots,\dimension\}$} }} \;
\nonl \For{$\coordinate=1,\dots,\dimension$}{
  \routinenl \lIf{$\coordinate{\,\in\,}\blockkw{\iter}{\sample}$}{$\xIk{\coordinate}{\iter+1} 
  = \scaledProx{\DimensionI{\coordinate} \StepsizekI{k}{\coordinate}}{\NnonsmoothIfunction{\coordinate}}{\xIk{\coordinate}{\iter}-(\DimensionI{\coordinate} \StepsizekI{k}{\coordinate})^{-1}(\grad\smoothI{\coordinate}{\xIk{\coordinate}{\iter}}+\Sigmak{\iter}\gradI{\coordinate} \constraint{\Zallk{\iter}})}
  $ \label{algorithm:genericxIk}
  }
  \routinenl \lElse{
  $ \xIk{\coordinate}{\iter+1} = \xIk{\coordinate}{\iter} $  \label{algorithm:genericxmIk}
  }
} 
\routinenl
$\Sallk{\iter+1}  
=   
\Zallk{\iter}    +   \big({\dualstepsizek{\iter}}/{\Sigmak{\iter}}\big)\, \Dimension \big(\xallk{\iter+1} -\xallk{\iter}\big) 
$ \; \label{algorithm:genericSallk}
\routinenl
$\Sigmak{\iter+1}  =   \Sigmak{\iter}    +   \dualstepsizek{\iter+1}$ \; \label{algorithm:Sigmak}
}
 \end{algorithm}
%
%
%
\smallskip

The equivalence between Algorithms~\ref{algorithm:newgenericprimaldual} and~\ref{algorithm:newgenericprimalonly} can be shown by introducing 
a dual variable $ \dualk{\iter} \equiv \Sigmak{\iter} (\Aall\Zallk{\iter}-\bvector)   $. 
The dual variable allows us to write, on Line~\ref{algorithm:genericxIk} of Algorithm~\ref{algorithm:newgenericprimalonly},
  \begin{equation} \label{primaltodual}
     \Sigmak{\iter}\gradi{\coordinate} \constraint{\Zallk{\iter}} =  \AI{\coordinate}^\transpose \Sigmak{\iter} (\Aall\Zallk{\iter}-\bvector) = \AI{\coordinate}^\transpose\dualk{\iter}  .
       \end{equation}
Line~\ref{algorithm:genericSallk} also gives us
  \begin{equation} \label{migrationone}
\Sigmak{\iter} (\Aall\Sallk{\iter+1}-\bvector)     
=     \dualk{\iter} + \dualstepsizek{\iter}     \Aall \Dimension (\xallk{\iter+1} -\xallk{\iter})  
.
\end{equation}
Algorithm~\ref{algorithm:newgenericprimalonly}-Line~\ref{algorithm:genericZallk} then yields  the dual update rule 
  \begin{equation} \label{migrationtwo}
  \nocolsep
  \begin{array}{rcl}
  \dualk{\iter+1}
  &=&
 \Sigmak{\iter+1} (\Aall \Zallk{\iter+1} -\bvector)
 \\
 &=& 
  \Aall ({\Sigmak{\iter}\Sallk{\iter+1} + \dualstepsizek{\iter+1}\xallk{\iter+1}})  -  \Sigmak{\iter+1} \bvector 
 \\
 &=& 
   \Sigmak{\iter}(\Aall \Sallk{\iter+1}  -   \bvector ) +  \dualstepsizek{\iter+1} (\Aall \xallk{\iter+1} -  \bvector ) 
 \\
&\refereq{\eqref{migrationone}}{=}& 
 %
 \dualk{\iter} + \dualstepsizek{\iter}     \Aall \Dimension (\xallk{\iter+1} -\xallk{\iter})   + \dualstepsizek{\iter+1} \altdualk{\iter+1}
,
\end{array}
\end{equation}
where a second dual variable $\altdualk{\iter}\equiv   (\Aall \xallk{\iter}  -  \bvector )$ has been introduced, which satisfies
\begin{equation}\label{altdualk}
 \altdualk{\iter+1} 
 =  \altdualk{\iter} + \Aall ( \xallk{\iter+1}  -  \xallk{\iter} )
 .
\end{equation}
Algorithm~\ref{algorithm:newgenericprimaldual}  can eventually be recovered after inclusion of~
\eqref{primaltodual}, \eqref{migrationtwo}, and~\eqref{altdualk} into Algorithm~\ref{algorithm:newgenericprimalonly}. 
Straightforward computations also lead to the expression previously given in~\eqref{ergodicaverage} for the sequence~$\Sallk{\iter}$, which can be derived by induction on~$\iter$. 
We note that the sequences~$\Sallk{\iter}$ and~$\Zallk{\iter}$ are bounded whenever the algorithm produces a bounded sequence~$\xallk{\iter}$.

In Sections~\ref{section:ESO} to~\ref{section:mainargument} we study the convergence of the sequences produced by \cref{algorithm:newgenericprimalonly}.
For analysis purposes, we consider the auxiliary sequence defined by
\begin{equation}\label{generichatxall}
\begin{array}{l}
\hatxallk{\iter+1} 
= 
 \scaledProx{\Dimension \Stepsizek{k}}{\Nnonsmoothfunction}{\xallk{\iter}-(\Dimension \Stepsizek{k})^{-1}(\grad\smooth{\xallk{\iter}}+\Sigmak{\iter}\grad \constraint{\Zallk{\iter}})}
,
\end{array}
\end{equation} 
which would coincide with iterate~$\hatxallk{\iter+1}$ if all coordinates were to be udpated at step~$\iter$ (i.e., in the event~$\blockk{\iter}=\{1,\dots,\dimension\}$).

\subsection{\acs{ESO} for block coordinate sampling}\label{section:ESO}

Let~$\filtrationk{\iter}\defeq 
\sigmaalgebrax{\xallk{0},\Sallk{0},\Zallk{0},\dots,\xallk{\iter},\Sallk{\iter},\Zallk{\iter}}
$ denote the sigma algebra of the process history up to step~$\iter$.
The upcoming property for Algorithm~\ref{algorithm:newgenericprimalonly} relates to the concept of \acdef{ESO}, as seen in \cite{RicTak14,RicTak16,fercoq15}.

 \begin{lemma}[Proximal step] \label{lemma:descentargumentcostnonsmooth}
In Algorithm~\ref{algorithm:newgenericprimalonly},  $\forall\xall\in\Xall$,
\begin{equation} \label{descentargument}
\Nnonsmooth{\hatxallk{\iter+1}}-\Nnonsmooth{\xall} \leq
 \inner{\grad\stronglyconvex{\hatxallk{\iter+1}}}{\xall-\hatxallk{\iter+1}} 
\leq 
\stronglyconvex{\xall} -\stronglyconvex{\hatxallk{\iter+1}} - \frac{1}{2} \normx{\xall-\hatxallk{\iter+1}}{\Dimension\Stepsizek{\iter}+\Convex}^2 
, 
\end{equation}
where 
\begin{equation}\label{stronglyconvex}
\nocolsep\begin{array}{l}
\stronglyconvex{\xall} = \smooth{\xallk{\iter}} +  \inner{\grad\smooth{\xallk{\iter}}}{\xall-\xallk{\iter}}
+
\Sigmak{\iter}\inner{\grad \constraint{\Zallk{\iter}}}{\xall-\Zallk{\iter}}
+ \frac{1}{2} \normx{\xall-\xallk{\iter}}{\Dimension\Stepsizek{\iter}}^2  
.
\end{array}
\end{equation}
\end{lemma}
\begin{proof}
 Equation~\eqref{generichatxall} rewrites as $\hatxallk{\iter+1}= \arg\min_{\altxall}\{\Nnonsmooth{\altxall}+\stronglyconvex{\altxall}\}$.
Hence, $0\in\partial\Nnonsmooth{\hatxallk{\iter+1}}+\grad\stronglyconvex{\hatxallk{\iter+1}}$, 
which yields the first inequality. The second inequality
follows by strong convexity of $\nonsmoothfunction+\stronglyconvexfunction$  with modulus $\Dimension\Stepsizek{\iter}+\Convex$.
\end{proof} 
%

\subsection{Extrapolation}\label{section:extrapolation}
The next lemma characterizes the sequence~$(\Sallk{\iter})$ as a convex combination of past primal iterates. It is an extension of Lemma~2 in \cite{fercoq15}. 

\begin{lemma}\label{lemma:extrapolation}
 In Algorithm~\ref{algorithm:newgenericprimalonly}, we have
\begin{equation}
\label{Sallk}
 \Sallk{\iter} = \sum_{\altiter=0}^{\iter} \gammacoefkl{\iter}{\altiter} \xallk{\altiter}
 , \qquad \iter\geq 1,
\end{equation}
where~$(\gammacoefkl{\iter}{\altiter})$ is a sequence of diagonal matrices such that $\gammacoefkl{1}{0}=
\identityI{\primaldimension}-\Dimension
$, $\gammacoefkl{1}{1}=\Dimension$ and, for~$\iter\geq 1$,
\begin{equation} \label{gammacoefkl}
 \nocolsep
\Sigmak{\iter} \gammacoefkl{\iter+1}{\altiter} \,{=} \left\{\!
 \begin{array}{ll}
 \Sigmak{\iter-1} \gammacoefkl{\iter}{\altiter}
 &\text{ for } \altiter\,{=}\, 0,\dots,\iter{-}1,
 \\
\dualstepsizek{\iter-1}\Dimension -\dualstepsizek{\iter}(\Dimension-\identityI{\primaldimension})
 &\text{ if } \altiter\,{=}\,\iter,
 \\
\dualstepsizek{\iter}\Dimension
 &\text{ if } \altiter\,{=}\,\iter{+}1.
 \end{array}\right.
\end{equation} 
Moreover,
$
  \gammacoefkl{\iter+1}{\iter} =  [\Sigmak{\iter-1}\gammacoefkl{\iter}{\iter} -\dualstepsizek{\iter}(\Dimension-\identityI{\primaldimension})]/\Sigmak{\iter}
$ 
and
$
\sum_{\altiter=0}^{\iter} \gammacoefkl{\iter}{\altiter}
=
\identityI{\primaldimension}
$.
\end{lemma}
\begin{proof}
  We proceed by induction.  By combining Lines~\ref{algorithm:genericZallk} and~\ref{algorithm:genericSallk} in Algorithm~\ref{algorithm:newgenericprimalonly}, we find
\begin{equation}\label{developSallk} \displaystyle
 \Sallk{\iter+1}  
=
\frac{1}{\Sigmak{\iter}} \left(
\Sigmak{\iter-1} \Sallk{\iter} +\dualstepsizek{\iter} \Dimension \xallk{\iter+1} -\dualstepsizek{\iter}(\Dimension-\identityI{\primaldimension})\xallk{\iter} \right),
\end{equation}
 which yields 
 $  \Sallk{1}  
=
(\identityI{\primaldimension}-\Dimension) \xallk{0} +\Dimension \xallk{1} 
$, and the values of $\gammacoefkl{1}{0}$ and~$\gammacoefkl{1}{1}$. Suppose now that~\eqref{Sallk} holds for $\iter\geq 1$, then we infer from~\eqref{developSallk} that
\begin{equation}\label{inspection}
 \nocolsep
 \begin{array}{rcl}
 \Sigmak{\iter}\Sallk{\iter+1}  
&=&
  \sum_{\altiter=0}^{\iter} \Sigmak{\iter-1} \gammacoefkl{\iter}{\altiter} \xallk{\altiter} +\dualstepsizek{\iter} \Dimension \xallk{\iter+1} -\dualstepsizek{\iter}(\Dimension-\identityI{\primaldimension})\xallk{\iter} 
\\
&=&
  \sum_{\altiter=0}^{\iter-1} \Sigmak{\iter-1} \gammacoefkl{\iter}{\altiter} \xallk{\altiter} + \Sigmak{\iter-1} \gammacoefkl{\iter}{\iter} \xallk{\iter} +\dualstepsizek{\iter} \Dimension \xallk{\iter+1} -\dualstepsizek{\iter}(\Dimension-\identityI{\primaldimension})\xallk{\iter}
\\
&=&
  \sum_{\altiter=0}^{\iter-1} \Sigmak{\iter-1} \gammacoefkl{\iter}{\altiter} \xallk{\altiter} + [ \Sigmak{\iter-1} \gammacoefkl{\iter}{\iter} - \dualstepsizek{\iter}(\Dimension-\identityI{\primaldimension})] \xallk{\iter} + \dualstepsizek{\iter} \Dimension \xallk{\iter+1} 
 .
 \end{array}
\end{equation}  
Equation~\eqref{gammacoefkl} follows by inspection of~\eqref{Sallk} and~\eqref{inspection}. Hence, \eqref{Sallk} holds for all $\iter$.

It can be checked that~$ \Sallk{\iter} $ is a convex combination of $\xallk{0},\dots,\xallk{\iter}$. 
Again, we do this by induction. We already know that 
$
\sum_{\altiter=0}^{1} \gammacoefkl{1}{\altiter}
= 
(\identityI{\primaldimension}-\Dimension) +\Dimension = \identityI{\primaldimension}
$
and, by~\eqref{gammacoefkl},
$
\sum_{\altiter=0}^{2} \gammacoefkl{2}{\altiter}
=
\big({1}/{\Sigmak{1}}\big) \big[ 
 \Sigmak{0} \gammacoefkl{1}{0}
+
\dualstepsizek{0}\Dimension -\dualstepsizek{1}(\Dimension-\identityI{\primaldimension})
+
\dualstepsizek{1}\Dimension
\big]
= 
\big({1}/{\Sigmak{1}}\big) \big[ 
 \Sigmak{0} +\dualstepsizek{1}
\big] \identityI{\primaldimension}
=
\identityI{\primaldimension}
$.
Suppose now that $\sum_{\altiter=0}^{\iter} \gammacoefkl{\iter}{\altiter}
=\identityI{\primaldimension}
$ is true for some $\iter\geq 1$, then
\[
\begin{array}{l}
\displaystyle
\sum\nolimits_{\altiter=0}^{\iter+1} \gammacoefkl{\iter+1}{\altiter}
\refereq{\eqref{gammacoefkl}}{=}
\big({1}/{\Sigmak{\iter}}\big) \left[
 \Sigmak{\iter-1} \sum\nolimits_{\altiter=0}^{\iter-1}  \gammacoefkl{\iter}{\altiter}
+ 
\dualstepsizek{\iter-1}\Dimension -\dualstepsizek{\iter}(\Dimension-\identityI{\primaldimension})
 +
\dualstepsizek{\iter}\Dimension
\right]
\quad\\\quad\hfill\displaystyle
=
\big({1}/{\Sigmak{\iter}}\big) \left[
 \Sigmak{\iter-1} (\identityI{\primaldimension} -  \gammacoefkl{\iter}{\iter})
+ 
\dualstepsizek{\iter-1}\Dimension +\dualstepsizek{\iter}\identityI{\primaldimension}
\right]
\quad\\\quad\hfill\displaystyle
\refereq{\eqref{gammacoefkl}}{=}
\big({1}/{\Sigmak{\iter}}\big) \left[
 \Sigmak{\iter-1} \left(\identityI{\primaldimension} -  \big(\dualstepsizek{\iter-1}/{\Sigmak{\iter-1}}\big){\Dimension }
\right)
+ 
\dualstepsizek{\iter-1}\Dimension +\dualstepsizek{\iter}\identityI{\primaldimension}
\right]
=
\big({1}/{\Sigmak{\iter}}\big) \left[
 \Sigmak{\iter-1}  +\dualstepsizek{\iter}
\right] \identityI{\primaldimension}
=
\identityI{\primaldimension}
,
\end{array}
\]
and the last identity holds by induction on~$\iter$.
Observe that this last result could also be obtained immediately by simple inspection of~\eqref{ergodicaverage}.
\end{proof}
Now, define $\vcompositefunction=(\compositeIfunction{1},\dots,\compositeIfunction{\dimension})$ and $\hatcompositek{\iter}=
\inner{\hatvcompositek{\iter}}{\onesvectorn{\dimension}}
$, where
\begin{equation} 
 \label{hatvcompositek}
 \hatvcompositek{\iter} = \sum_{\altiter=0}^{\iter} \gammacoefkl{\iter}{\altiter} \vcomposite{\xallk{\altiter}}
 , \qquad \iter\geq 1.
\end{equation}
By convexity, it follows from~\eqref{Sallk} and~\eqref{hatvcompositek} that  $\hatvcompositek{\iter}\geq\vcomposite{\Sallk{\iter}}$ and  $\hatcompositek{\iter}\geq\composite{\Sallk{\iter}}$.
We show the following result.

\newcommand{\genericmatrix}{M}%
\newcommand{\genericmatrixi}[1]{\genericmatrix_{#1}}%
\begin{lemma} In  Algorithm~\ref{algorithm:newgenericprimalonly} and for any positive definite $\genericmatrix=\diag(\genericmatrixi{1},\dots,\genericmatrixi{\dimension})$,
\begin{align}
& 
  \condexpectation{\filtrationk{\iter}}{\normx{\xallk{\iter+1}-\xallsol}{\Dimension\genericmatrix}^2} 
=
\normx{\hatxallk{\iter+1}-\xallsol}{\genericmatrix}^2  
+ \normx{\xallk{\iter}-\xallsol}{(\Dimension-\identityI{\primaldimension})\genericmatrix}^2
,
 \label{expectationdistance}
\\
&\label{expectationhatcomposite} \textstyle
\condexpectation{\filtrationk{\iter}}{\hatcompositek{\iter+1}} 
=
\frac{1}{\Sigmak{\iter}} [\Sigmak{\iter-1}\hatcompositek{\iter} +\dualstepsizek{\iter} \composite{\hatxallk{\iter+1}}]
,
\end{align}
where~$\hatxallk{\iter+1}$ and~$\hatcompositek{\iter}$ are defined as in~\eqref{generichatxall} and~\eqref{hatvcompositek}. 
\end{lemma}
\begin{proof}
For $\coordinate=1,\dots,\dimension$, 
We find,
\begin{equation}\nocolsep
\begin{array}{rcl}
\condexpectation{\filtrationk{\iter}}{\normx{\xIk{\coordinate}{\iter+1}-\xIsol{\coordinate}}{\DimensionI{\coordinate}\genericmatrixi{\coordinate}}} 
&=&
\probai{\coordinate}\normx{\hatxIk{\coordinate}{\iter+1}-\xIsol{\coordinate}}{\DimensionI{\coordinate}\genericmatrixi{\coordinate}}^2+ (1-\probai{\coordinate})
  \normx{\xIk{\coordinate}{\iter}-\xIsol{\coordinate}}{\DimensionI{\coordinate}\genericmatrixi{\coordinate}}^2
\\&=&
\normx{\hatxIk{\coordinate}{\iter+1}-\xIsol{\coordinate}}{\genericmatrixi{\coordinate}}^2+ (\probai{\coordinate}^{-1}-1)
  \normx{\xIk{\coordinate}{\iter}-\xIsol{\coordinate}}{\genericmatrixi{\coordinate}}^2.
\end{array}
\end{equation}
Summing up the above for $\coordinate=1,\dots,\dimension$ gives~\eqref{expectationdistance}.
Next, observe that
$
\condexpectation{\filtrationk{\iter}}{\compositeI{\coordinate}{\xallk{\iter+1}}}
 =
 \probai{\coordinate} \compositeI{\coordinate}{\hatxallk{\iter+1}} +(1-\probai{\coordinate}) \compositeI{\coordinate}{\xallk{\iter}} 
 $ for $\coordinate\in\{1,\dots,\dimension\}$,
 which in matrix form rewrites as
\begin{equation} \label{expectationvcomposite}
 \condexpectation{\filtrationk{\iter}}{\vcomposite{\xallk{\iter+1}}}
 =
 \Dimension^{-1} \vcomposite{\hatxallk{\iter+1}} +(\identityI{\primaldimension}- \Dimension^{-1}) \vcomposite{\xallk{\iter}} .
\end{equation}
It follows that
\begin{equation}\label{expectationhatvcomposite}
 \nocolsep
 \begin{array}{ll}
  \condexpectation{\filtrationk{\iter}}{\hatvcompositek{\iter+1}}
  \hspace{-15mm}&\hspace{15mm} 
  \refereq{\eqref{hatvcompositek}}{=}
  \sum_{\altiter=0}^{\iter-1} \gammacoefkl{\iter+1}{\altiter} \vcomposite{\xallk{\altiter}}+ \gammacoefkl{\iter+1}{\iter} \vcomposite{\xallk{\iter}}
  +\gammacoefkl{\iter+1}{\iter+1} \condexpectation{\filtrationk{\iter}}{\vcomposite{\xallk{\iter+1}}}
  \\ &\refereq{\eqref{gammacoefkl}}{=}
  \big(\frac{\Sigmak{\iter-1}}{\Sigmak{\iter}}\big)\sum_{\altiter=0}^{\iter-1} \gammacoefkl{\iter}{\altiter} \vcomposite{\xallk{\altiter}}+ \gammacoefkl{\iter+1}{\iter} \vcomposite{\xallk{\iter}}
  + \big(\frac{\dualstepsizek{\iter}}{\Sigmak{\iter}}\big)\Dimension \condexpectation{\filtrationk{\iter}}{\vcomposite{\xallk{\iter+1}}}
  \\ &\refereq{\eqref{expectationvcomposite}}{=}
  \big(\frac{\Sigmak{\iter-1}}{\Sigmak{\iter}}\big) \sum_{\altiter=0}^{\iter-1} \gammacoefkl{\iter}{\altiter} \vcomposite{\xallk{\altiter}}
+ [\gammacoefkl{\iter+1}{\iter} + (\frac{\dualstepsizek{\iter}}{\Sigmak{\iter}})   (\Dimension-\identityI{\primaldimension})] \vcomposite{\xallk{\iter}} +\big(\frac{\dualstepsizek{\iter}}{\Sigmak{\iter}}\big) \vcomposite{\hatxallk{\iter+1}}
  \\ 
&\refereq{\eqref{hatvcompositek}}{=}
 \big(\frac{\Sigmak{\iter-1}}{\Sigmak{\iter}})\big)\hatvcompositek{\iter} +\big(\frac{\dualstepsizek{\iter}}{\Sigmak{\iter}}\big) \vcomposite{\hatxallk{\iter+1}}
 ,
 \end{array}
\end{equation}
which yields~\eqref{expectationhatcomposite} since $\compositefunction=
\inner{\vcompositefunction}{\onesvectorn{\dimension}}
$ and $\hatcompositek{\iter}=
\inner{\hatvcompositek{\iter}}{\onesvectorn{\dimension}}
$.
\end{proof}

\subsection{Properties of~$\constraintfunction$}\label{section:constraintfunction}

In this section we report useful properties related to the penalty function~$\constraintfunction$.
First notice that, for all~$\xall,\altxall,\altaltxall\in\Real^q$,
\begin{equation}
\label{genericidentity}
\nocolsep
\begin{array}{rcl}
\inner{\grad\constraint{\xall}}{\altxall-\altaltxall}
&=&
\inner{\grad\constraint{\xall}}{\altxall-\xall}
-
\inner{\grad\constraint{\xall}}{\altaltxall-\xall}
\\ 
&=&
\constraint{\altxall}-\constraint{\xall}-\frac{1}{2}\norm{\Aall(\altxall-\xall)}^2
-
[\constraint{\altaltxall}-\constraint{\xall}-\frac{1}{2}\norm{\Aall(\altaltxall-\xall)}^2 ]
\\
&=&
\constraint{\altxall}-\constraint{\altaltxall} -\frac{1}{2}\norm{\Aall(\altxall-\xall)}^2
+\frac{1}{2}\norm{\Aall(\altaltxall-\xall)}^2
.
\end{array}
\end{equation}
%
For all~$\xall,\altxall\in\Real^q$, and $\alpha\in[0,1]$, we also have
\begin{equation} \label{quadraticresult}
 \constraint{\alpha{\xall}+(1-\alpha)\altxall}
 =
 \alpha \constraint{\xall}
 +
 (1-\alpha) \constraint{\altxall}
 -
 \frac{\alpha (1-\alpha)}{2} \norm{\Aall(\xall-\altxall)}^2
 .
\end{equation}
For any $\xallsol\in\arg\min_{\xall}\constraint{\xall}$, we have $\grad\constraint{\xallsol} = 0$, and~$\constraintfunction$ satisfies
\begin{equation}\label{constraintidentity}
 \constraint{\xall} - \constraint{\xallsol}
 =
\frac{1}{2} \inner{\Aall^\transpose\Aall(\xall-\xallsol)}{\xall-\xallsol} 
 =
 \frac{1}{2} \norm{\Aall(\xall-\xallsol)}^2
 ,
 \quad
 \forall \xall\in\Real^q
 .
\end{equation}
Besides, we know from Line~\ref{algorithm:genericZallk} in Algorithm~\ref{algorithm:newgenericprimalonly} that
\begin{equation}\label{gradatZallk}
\nocolsep
\begin{array}{l}
 \grad \constraint{\Zallk{\iter}} 
 = 
 \grad \constraint {({\Sigmak{\iter-1}\Sallk{\iter} + \dualstepsizek{\iter}\xallk{\iter}})/\Sigmak{\iter}}
 =
\Aall^\transpose[\Aall({\Sigmak{\iter-1}\Sallk{\iter} + \dualstepsizek{\iter}\xallk{\iter}})/\Sigmak{\iter}-\bvector]
\qquad\qquad\qquad\\ \hfill
 =
\frac{1}{\Sigmak{\iter}} [\Sigmak{\iter-1} \Aall^\transpose(\Aall\Sallk{\iter}
-\bvector)+\dualstepsizek{\iter} \Aall^\transpose(\Aall \xallk{\iter}-\bvector)]
 =
\frac{1}{\Sigmak{\iter}} (\Sigmak{\iter-1} \grad \constraint{\Sallk{\iter}}
+\dualstepsizek{\iter} \grad \constraint{\xallk{\iter}})
,
\end{array}
\end{equation}
and it follows from Lines~\ref{algorithm:genericZallk} and~\ref{algorithm:genericSallk} in Algorithm~\ref{algorithm:newgenericprimalonly} that
\begin{equation}\label{constraintSallk}
\nocolsep
\begin{array}{l} 
\constraint{\Sallk{\iter+1}}
=
\Constraint{   \big(\frac{\Sigmak{\iter-1}}{\Sigmak{\iter}}\big)\Sallk{\iter} + \big(\frac{\dualstepsizek{\iter}}{\Sigmak{\iter}}\big)[\xallk{\iter}    +   \Dimension (\xallk{\iter+1} -\xallk{\iter})]}
 \\ \
 \refereq{\eqref{quadraticresult}}{=}
 \big(\frac{\Sigmak{\iter-1}}{\Sigmak{\iter}}\big) \constraint{\Sallk{\iter}}
 +
 \big(\frac{\dualstepsizek{\iter}}{\Sigmak{\iter}}\big) \constraint{\xallk{\iter}    +   \Dimension (\xallk{\iter+1} -\xallk{\iter})}
 -
 \frac{1}{2}\big(\frac{\dualstepsizek{\iter}\Sigmak{\iter-1}}{(\Sigmak{\iter})^2}\big) \norm{\Aall(\xallk{\iter}    +   \Dimension (\xallk{\iter+1} -\xallk{\iter})-\Sallk{\iter})}^2 
 .
\end{array}
\end{equation}
%
The last result of this section requires new notions to be introduced. For $\coordinate\in\{1,\ldots,\dimension\}$ let $\Uniti{\coordinate}$ be the $\primaldimension\times \primaldimension$ block matrix defined by $\Uniti{\coordinate}=\diag(0,\ldots,\identityI{\primaldimensioni{\coordinate}},0,\ldots,0)$. Clearly $\sum_{\coordinate=1}^{\dimension}\Uniti{\coordinate}=\identityI{\primaldimension}$, and $\Uniti{\coordinate}\xall=(0,\ldots,\xI{\coordinate},\ldots,0)$ for any  $\xall=(\xI{1},\ldots,\xI{\dimension})\in\Real^{\primaldimension}$. Then, for any random block $\block\in\Block$, we define the $\primaldimension\times \primaldimension$ matrix $\UnitI{\block}=\sum_{\coordinate\in\block}\Uniti{\coordinate}.$ We find $\expectation{\UnitI{\block}\Dimension}=\identityI{\primaldimension}$ and $\constraintcovariance = \expectation{\UnitI{\block}\Dimension\Aall^\transpose\Aall\Dimension\UnitI{\block}}$, where the expectations are taken with respect to the probabbility measure~$\probabilitymeasure$.
It follows that

\OBoff{Fix the notation~$\block$, which is sometimes seen as a random variable. I was never happy with~$\block$ and I will think of a solution. It is difficult because~$i$ is a single coordinated index, $I$  is the identity matrix, $J$  looks like a random variable, and $J_k$ leads to double capital subscripts. Maybe use $i$ for coordinate directions, $\iota$ or $j$ for coordinate direction sets, and $\text{I}$ for the random variable. }

%
\begin{align}   
&
\frac{1}{2}\norm{\Aall(\hatxallk{\iter+1}-\Sallk{\iter})}^2
=
\frac{1}{2}\norm{\condexpectation{\filtrationk{\iter}}{\Aall(\xallk{\iter}+\Dimension(\xallk{\iter+1}-\xallk{\iter})-\Sallk{\iter})}}^2
\nonumber\quad\\& \ 
=
\frac{1}{2}\condexpectation{\filtrationk{\iter}}{\norm{\Aall(\xallk{\iter}{+}\,\Dimension(\xallk{\iter\,{+}\,1}{-}\,\xallk{\iter}){-}\,\Sallk{\iter})}^2} 
{-}\, \frac{1}{2}\condexpectation{\filtrationk{\iter}}{\norm{\Aall(\Dimension\UnitI{\block}\,{-}\,\identity)(\hatxallk{\iter\,{+}\,1}{-}\,\xallk{\iter})}^2}
\nonumber\quad\\& \ 
=
\frac{1}{2}\condexpectation{\filtrationk{\iter}}{\norm{\Aall(\xallk{\iter}+\Dimension(\xallk{\iter+1}-\xallk{\iter})-\Sallk{\iter})}^2} 
- \frac{1}{2}\normx{\hatxallk{\iter+1}-\xallk{\iter}}{\constraintcovariance-\Aall^\transpose\Aall}^2
\nonumber\quad\\& \ 
\refereq{\eqref{constraintSallk}}{=}
{\Big(}\frac{\Sigmak{\iter}}{\dualstepsizek{\iter}}{\Big)}  \constraint{\Sallk{\iter}}
- {\Big(}\frac{(\Sigmak{\iter})^2}{\dualstepsizek{\iter}\Sigmak{\iter-1}}{\Big)} \condexpectation{\filtrationk{\iter}}{\constraint{\Sallk{\iter+1}}}  
+ {\Big(}\frac{\Sigmak{\iter}}{\Sigmak{\iter-1}}{\Big)} \condexpectation{\filtrationk{\iter}}{\constraint{\xallk{\iter}    +   \Dimension (\xallk{\iter+1} -\xallk{\iter})}}
\nonumber\\& \pushright{
- \frac{1}{2} \normx{\hatxallk{\iter+1}-\xallk{\iter}}{\constraintcovariance-\Aall^\transpose\Aall}^2
}
\nonumber\quad\\& \ 
\refereq{\eqref{constraint}}{=}
{\Big(}\frac{\Sigmak{\iter}}{\dualstepsizek{\iter}}{\Big)}  \constraint{\Sallk{\iter}}
- {\Big(}\frac{(\Sigmak{\iter})^2}{\dualstepsizek{\iter}\Sigmak{\iter-1}}{\Big)} \condexpectation{\filtrationk{\iter}}{\constraint{\Sallk{\iter+1}}}  
- \frac{1}{2} \normx{\hatxallk{\iter+1}-\xallk{\iter}}{\constraintcovariance-\Aall^\transpose\Aall}^2
\nonumber\quad\\& \pushright{
+
\frac{1}{2} {\Big(}\frac{\Sigmak{\iter}}{\Sigmak{\iter-1}}{\Big)} \condexpectation{\filtrationk{\iter}}{\norm{[\Aall\hatxallk{\iter+1}-\bvector]+\Aall[\xallk{\iter}    +   \Dimension (\xallk{\iter+1} -\xallk{\iter})-\hatxallk{\iter+1}] }^2}
}
\nonumber\quad\\& \ 
=
{\Big(}\frac{\Sigmak{\iter}}{\dualstepsizek{\iter}}{\Big)}  \constraint{\Sallk{\iter}}
- {\Big(}\frac{(\Sigmak{\iter})^2}{\dualstepsizek{\iter}\Sigmak{\iter-1}}{\Big)} \condexpectation{\filtrationk{\iter}}{\constraint{\Sallk{\iter+1}}}  
- \frac{1}{2} \normx{\hatxallk{\iter+1}-\xallk{\iter}}{\constraintcovariance-\Aall^\transpose\Aall}^2
\nonumber\\& \pushright{
+ {\Big(}\frac{\Sigmak{\iter}}{\Sigmak{\iter-1}}{\Big)} \big[ \constraint{\hatxallk{\iter+1}} + \frac{1}{2}\condexpectation{\filtrationk{\iter}}{\norm{\Aall(\Dimension\UnitI{\block}-\identity) (\hatxallk{\iter+1} -\xallk{\iter}) }^2}\big]
}
\nonumber\quad\\& \ 
= \!
{\Big(}\frac{\Sigmak{\iter}}{\dualstepsizek{\iter}}{\Big)} \constraint{\Sallk{\iter}}
\,{-}\, 
{\Big(}\frac{(\Sigmak{\iter})^2}{\dualstepsizek{\iter}\Sigmak{\iter-1}}{\!\Big)} \condexpectation{\filtrationk{\iter}}{\constraint{\Sallk{\iter+1}}}  
\,{+}\, 
\frac{1}{2}{\Big(}\frac{\dualstepsizek{\iter}}{\Sigmak{\iter-1}}{\!\Big)} \normx{\hatxallk{\iter+1}-\xallk{\iter}}{\constraintcovariance-\Aall^\transpose\Aall}^2
\,{+}\, 
{\Big(}\frac{\Sigmak{\iter}}{\Sigmak{\iter-1}}{\!\Big)} \constraint{\hatxallk{\iter+1}} 
,
\nonumber\\&
\label{variance}
\end{align} 
where we have used $\condexpectation{\filtrationk{\iter}}{\Aall[\xallk{\iter}    +   \Dimension (\xallk{\iter+1} -\xallk{\iter})-\hatxallk{\iter+1}] }=0$.

\subsection{Main descent argument}\label{section:mainargument}
%
We now derive the main argument for the theorems.
Let~$\xallsol$ be a solution of~\eqref{genericP}. It follows from Lemma~\ref{lemma:descentargumentcostnonsmooth} that
\allowdisplaybreaks
\begin{align}
 \nonumber&
\composite{\hatxallk{\iter+1}} -\compositeopt
\\\nonumber& \
\refereq{\eqref{smoothnessconvexity}}{\leq} 
 \Nnonsmooth{\hatxallk{\iter+1}} +
\smooth{\xallk{\iter}} -\compositeopt +\inner{\grad\smooth{\xallk{\iter}}}{\hatxallk{\iter+1}-\xallk{\iter}} + \frac{1}{2}\normx{\hatxallk{\iter+1}-\xallk{\iter}}{\NSmoothness}^2
\\\nonumber& \
\refereq{\eqref{stronglyconvex}}{=} [\Nnonsmooth{\hatxallk{\iter+1}}+\stronglyconvex{\hatxallk{\iter+1}}  ] -\compositeopt
 - \Sigmak{\iter}\inner{\grad \constraint{\Zallk{\iter}}}{\hatxallk{\iter+1}-\Zallk{\iter}} 
 - \frac{1}{2} \normx{\hatxallk{\iter+1}-\xallk{\iter}}{\Dimension\Stepsizek{\iter}-\NSmoothness}^2 
%
%
\\\nonumber& \
\refereq{\eqref{descentargument}}{\leq} 
[\Nnonsmooth{\xallsol}{\,+\,}\stronglyconvex{\xallsol} {\,-\,} \frac{1}{2} \normx{\hatxallk{\iter{+}1}{-}\xallsol}{\Dimension\Stepsizek{\iter}{+}\Convex}^2 ] {\,-\,} \compositeopt 
{\,-\,} \Sigmak{\iter}\inner{\grad \constraint{\Zallk{\iter}}}{\hatxallk{\iter{+}1} {-}\Zallk{\iter}}  
{\,-\,} \frac{1}{2} \normx{\hatxallk{\iter{+}1}{-}\xallk{\iter}}{\Dimension\Stepsizek{\iter}{-}\NSmoothness}^2 
%
\\\nonumber& \
\refereq{\eqref{stronglyconvex}}{=}  
\smooth{\xallk{\iter}}-\smooth{\xallsol}
+  \inner{\grad\smooth{\xallk{\iter}}}{\xallsol-\xallk{\iter}} 
+ \Sigmak{\iter}\inner{\grad \constraint{\Zallk{\iter}}}{\xallsol-\hatxallk{\iter+1}}
+ \frac{1}{2} \normx{\xallk{\iter}-\xallsol}{\Dimension\Stepsizek{\iter}}^2  
\\\nonumber&\pushright{
- \frac{1}{2} \normx{\hatxallk{\iter+1}-\xallsol}{\Dimension\Stepsizek{\iter}+\Convex}^2  - \frac{1}{2} \normx{\hatxallk{\iter+1}-\xallk{\iter}}{\Dimension\Stepsizek{\iter}-\NSmoothness}^2  
}
\\\nonumber& \
\refereq{\eqref{smoothnessconvexity}}{\leq} 
\Sigmak{\iter}\inner{\grad \constraint{\Zallk{\iter}}}{\xallsol-\hatxallk{\iter+1}} 
+ \frac{1}{2} \normx{\xallk{\iter}-\xallsol}{\Dimension\Stepsizek{\iter}}^2  
- \frac{1}{2} \normx{\hatxallk{\iter+1}-\xallsol}{\Dimension\Stepsizek{\iter}+\Convex}^2 
- \frac{1}{2} \normx{\hatxallk{\iter+1}-\xallk{\iter}}{\Dimension\Stepsizek{\iter}-\NSmoothness}^2  
%
\\\nonumber& \ 
\refereq{\eqref{gradatZallk}}{=}
\Sigmak{\iter-1} \inner{  \grad \constraint{\Sallk{\iter}}}{\xallsol-\hatxallk{\iter+1}}  
+ \dualstepsizek{\iter}\inner{  \grad \constraint{\xallk{\iter}}}{\xallsol-\hatxallk{\iter+1}}  
+ \frac{1}{2} \normx{\xallk{\iter}-\xallsol}{\Dimension\Stepsizek{\iter}}^2  
\\\nonumber&\pushright{
- \frac{1}{2} \normx{\hatxallk{\iter+1}-\xallsol}{\Dimension\Stepsizek{\iter}+\Convex}^2  
- \frac{1}{2} \normx{\hatxallk{\iter+1}-\xallk{\iter}}{\Dimension\Stepsizek{\iter}\NSmoothness}^2  
}
\\\nonumber& \
\refereq{\eqref{genericidentity}}{=}
-\Sigmak{\iter}(\constraint{\hatxallk{\iter+1}}-\constraintmin) -\frac{\Sigmak{\iter-1}}{2}\norm{\Aall(\xallsol-\Sallk{\iter})}^2
+\frac{\Sigmak{\iter-1}}{2}\norm{\Aall(\hatxallk{\iter+1}-\Sallk{\iter})}^2  
-\frac{\dualstepsizek{\iter}}{2}\norm{\Aall(\xallsol-\xallk{\iter})}^2
\\\nonumber&\pushright{
+\frac{\dualstepsizek{\iter}}{2}\norm{\Aall(\hatxallk{\iter+1}-\xallk{\iter})}^2  
+ \frac{1}{2} \normx{\xallk{\iter}-\xallsol}{\Dimension\Stepsizek{\iter}}^2  
- \frac{1}{2} \normx{\hatxallk{\iter+1}-\xallsol}{\Dimension\Stepsizek{\iter}+\Convex}^2  - \frac{1}{2} \normx{\hatxallk{\iter+1}-\xallk{\iter}}{\Dimension\Stepsizek{\iter}-\NSmoothness}^2  
}
\\\nonumber& \
\refereq{\eqref{constraint}}{=}
-[\Sigmak{\iter}\constraint{\hatxallk{\iter+1}} + \Sigmak{\iter-1}\constraint{\Sallk{\iter}}+\dualstepsizek{\iter} \constraint{\xallk{\iter}}-2\Sigmak{\iter}\constraintmin]
+\frac{\Sigmak{\iter-1}}{2}\norm{\Aall(\hatxallk{\iter+1}-\Sallk{\iter})}^2 
\\\nonumber&\pushright{
+ \frac{1}{2} \normx{\xallk{\iter}-\xallsol}{\Dimension\Stepsizek{\iter}}^2  - \frac{1}{2} \normx{\hatxallk{\iter+1}-\xallsol}{\Dimension\Stepsizek{\iter}+\Convex}^2 
- \frac{1}{2} \normx{\hatxallk{\iter+1}-\xallk{\iter}}{\Dimension\Stepsizek{\iter}-\dualstepsizek{\iter}\Aall^\transpose\Aall-\NSmoothness}^2  }
\\\nonumber& \
\refereq{\eqref{variance}}{=}
\frac{1}{2} \normx{\xallk{\iter}-\xallsol}{\Dimension\Stepsizek{\iter}}^2  + \big(\frac{(\Sigmak{\iter-1})^2}{\dualstepsizek{\iter}}\big) ( \constraint{\Sallk{\iter}} -\constraintmin) - \frac{1}{2} \normx{\hatxallk{\iter+1}-\xallsol}{\Dimension\Stepsizek{\iter}+\Convex}^2 
\\\nonumber&\pushright{
- \big(\frac{(\Sigmak{\iter})^2}{\dualstepsizek{\iter}}\big) \condexpectation{\filtrationk{\iter}}{\constraint{\Sallk{\iter+1}}-\constraintmin}  - \frac{1}{2} \normx{\hatxallk{\iter+1}-\xallk{\iter}}{\Dimension\Stepsizek{\iter}-\dualstepsizek{\iter}\constraintcovariance-\NSmoothness}^2  -\dualstepsizek{\iter} (\constraint{\xallk{\iter}}-\constraintmin) }
.
\end{align}
%
Using~\eqref{expectationdistance} and~\eqref{expectationhatcomposite},  setting $\Stepsizek{\iter}=\Stepsize/\stepsizek{\iter}$, and multiplying by~$\dualstepsizek{\iter}$, we find
\begin{equation} \label{earlystronglyconvexfejer}
\nocolsep
\begin{array}{l}
\Condexpectation{\filtrationk{\iter}}{ \frac{\dualstepsizek{\iter}}{2\stepsizek{\iter}}    \normx{\xallk{\iter+1}-\xallsol}{\Dimension^2\Stepsize+\stepsizek{\iter}\Dimension\Convex}^2 
  +   \Sigmak{\iter}
  \penalizedk{\iter+1}
  }
 \leq  \frac{\dualstepsizek{\iter}}{2\stepsizek{\iter}} \normx{\xallk{\iter}-\xallsol}{\Dimension^2\Stepsize+\stepsizek{\iter}(\Dimension-\identityI{\primaldimension})\Convex}^2  +  \Sigmak{\iter-1}
 \penalizedk{\iter}
\ \\\hfill
  - \frac{\dualstepsizek{\iter}}{2\stepsizek{\iter}} \normx{\hatxallk{\iter+1}-\xallk{\iter}}{\Dimension\Stepsize-\stepsizek{\iter}(\dualstepsizek{\iter}\constraintcovariance+\NSmoothness)}^2  
- (\dualstepsizek{\iter})^2 (\constraint{\xallk{\iter}}-\constraintmin)
,
\end{array} 
\end{equation} 
\interdisplaylinepenalty=10000
where we define $\penalizedk{\iter}= \hatcompositek{\iter} -\compositeopt
 + \Sigmak{\iter-1} (\constraint{\Sallk{\iter}}-\constraintmin ) $.
For the convergence of the above sequence we require, at every~$\iter$, 
\begin{subequations} \label{stepsizecondition}
\begin{align}
 \label{stepsizeconditionone}
&
\stepsizek{\iter}(\dualstepsizek{\iter}\constraintcovariance+\NSmoothness) 
- 
\Dimension\Stepsize
\preceq
0
,
 \\  \label{stepsizeconditiontwo}
&
\frac{\dualstepsizek{\iter+1}}{\stepsizek{\iter+1}} (\Dimension^2\Stepsize+\stepsizek{\iter+1}\fixed{(\Dimension-\identityI{\primaldimension})}\Convex)
-\frac{\dualstepsizek{\iter}}{\stepsizek{\iter}}(\Dimension^2\Stepsize +\stepsizek{\iter}\Dimension\Convex) 
\preceq 
0
.
\end{align}
\end{subequations}
%
%
After introducing 
\[
\lyapunovk{\iter}{\xall} = \frac{\dualstepsizek{\iter}}{2\stepsizek{\iter}}    \normx{\xallk{\iter}-\xall}{\Dimension^2\Stepsize+\stepsizek{\iter}\Dimension\Convex}^2   +  \Sigmak{\iter-1}  \penalizedk{\iter}
,
\]
we can rewrite~\eqref{earlystronglyconvexfejer} as the inequality
\begin{equation} \label{earlystronglyconvexfejersimple}
\Condexpectation{\filtrationk{\iter}}{\lyapunovk{\iter+1}{\xallsol} 
  }
\leq  \lyapunovk{\iter}{\xallsol}
  - \frac{\dualstepsizek{\iter}}{2\stepsizek{\iter}} \normx{\hatxallk{\iter+1}-\xallk{\iter}}{\Dimension\Stepsize-\stepsizek{\iter}(\dualstepsizek{\iter}\constraintcovariance+\NSmoothness)}^2  
-(\dualstepsizek{\iter})^2 (\constraint{\xallk{\iter}}-\constraintmin)
\,
,
\end{equation} 
which becomes the focal point in our convergence analysis.
Moreover, taking the total expectation in~\eqref{earlystronglyconvexfejersimple} and iterating on~$\iter$ gives
\begin{equation} \label{earlystronglyconvexfejertotal}
\Expectation{ 
\lyapunovk{\iter}{\xallsol}  
+
\sum\nolimits_{\altiter=0}^{\iter-1}\frac{\dualstepsizek{\altiter}}{2\stepsizek{\altiter}} \normx{\hatxallk{\iter}-\xallk{\altiter}}{\Dimension\Stepsize-\stepsizek{\altiter}(\dualstepsizek{\altiter}\constraintcovariance+\NSmoothness)}^2
}
\leq  \constantC
%
.
\end{equation} 
where $\constantC\defeq \frac{\dualstepsizek{0}}{2\stepsizek{0}} \normx{\xallk{0}-\xallsol}{\Dimension^2\Stepsize+\stepsizek{0}(\Dimension-\identityI{\primaldimension})\Convex}^2 $.
It follows from $\hatcompositek{\iter}\geq\composite{\Sallk{\iter}}$ and the definitions of~$\lyapunovk{\iter}{\xallsol}$ and~$\penalizedk{\iter}$, that
\begin{equation} \label{fortyeight}
 \Expectation{ \frac{\dualstepsizek{\iter}}{2\stepsizek{\iter}\Sigmak{\iter-1}}  \normx{\xallk{\iter}-\xallsol}{\Dimension^2\Stepsize+\stepsizek{\iter}\Dimension\Convex}^2  } +   \Sigmak{\iter-1} \expectation{ \constraint{\Sallk{\iter}}-\constraintmin }  +    \expectation{ \composite{\Sallk{\iter}} -\compositeopt} 
\leq  
\frac{\constantC}{\Sigmak{\iter-1}} 
.
\end{equation} 
Besides, if~$\compositefunction$ is bounded from below by a finite constant~$\compositemin$, 
then we also find
\begin{equation} \label{fortyeightbis}
\expectation{ \constraint{\Sallk{\iter}}-\constraintmin }   
\leq  
\frac{\constantC}{(\Sigmak{\iter-1})^2}
+\frac{\compositeopt -
    \compositemin}{\Sigmak{\iter-1}}
.
\end{equation}

Recall the sequence~\eqref{earlystronglyconvexfejersimple}. Convergence of $(\lyapunovk{\iter}{\xall})$ follows almost surely from~\eqref{earlystronglyconvexfejersimple} on condition that~$\lyapunovk{\iter}{\xall}$ is bounded from below, in which case Doob's supermartingale theorem applies. As we see below, such a lower bound exists in the case when strong duality holds in~\eqref{genericP} with respect to the constraint $\grad \constraint{\xall}=\Aall^\transpose(\Aall\xall-\bvector)=0$ and an optimal multiplier can be found (\cref{section:withmultiplier}).
In the negative, it is necessary to step back and to consider the convergence of a downscaled version of the sequence~$\lyapunovk{\iter}{\xall}$ (\cref{section:withoutmultiplier}).

\subsection{Duality and existence of a Lagrange multiplier} \label{section:withmultiplier}
Rewrite the constraint in~\eqref{genericP} as $\grad\constraint{\xall} = 0$, i.e., $  \Aall^\transpose\Aall\xall-\Aall^\transpose\bvector = 0$. First assume that~\eqref{genericP}  admits an optimal primal-dual pair $(\xallsol,\multipliersol)$ satisfying  $ 0 \in \partial \composite{\xallsol} +\Aall^\transpose\Aall\multipliersol $, or

\begin{equation}\label{boundcompositeSallk}
\begin{array}{l}
 \composite{\Sallk{\iter}}-\compositeopt
 =
 \composite{\Sallk{\iter}}-\composite{\xallsol}
 \geq
 \inner{-\Aall^\transpose\Aall\multipliersol}{\Sallk{\iter}-\xallsol}
 =
 -\inner{\Aall\multipliersol}{\Aall(\Sallk{\iter}-\xallsol)}
 \ \\\hfill
 \geq
 -\norm{\Aall\multipliersol}\norm{\Aall(\Sallk{\iter}-\xallsol)}
 \refereq{\eqref{constraintidentity}}{=}
  -\norm{\Aall\multipliersol}\sqrt{2(\constraint{\Sallk{\iter}}-\constraintmin)}
 ,
 \end{array}
\end{equation}
where we have used $\Aall^\transpose(\Aall\xallsol-\bvector)=0$.
Since $\hatcompositek{\iter}\geq\composite{\Sallk{\iter}}$, we find
\begin{equation}\label{boundkpenalizedkSC}
\begin{array}{l}
 \Sigmak{\iter-1}\penalizedk{\iter} 
 =
 \Sigmak{\iter-1} [\hatcompositek{\iter} -\compositeopt
 + \Sigmak{\iter-1} (\constraint{\Sallk{\iter}}-\constraintmin )  ]
%
\geq
 \Sigmak{\iter-1} [\composite{\Sallk{\iter}} -\compositeopt
 + \Sigmak{\iter-1} (\constraint{\Sallk{\iter}}-\constraintmin )  ]
 \quad \\ \hfill
 \refereq{\eqref{boundcompositeSallk}}{\geq} 
   (\Sigmak{\iter-1})^2 (\constraint{\Sallk{\iter}}-\constraintmin )   -\sqrt{2}\norm{\Aall\multipliersol}\sqrt{(\Sigmak{\iter-1})^2(\constraint{\Sallk{\iter}}-\constraintmin)} 
\quad \\ \hfill
=
\frac{1}{2}\,(\Sigmak{\iter-1})^2 (\constraint{\Sallk{\iter}}-\constraintmin )   
+
\left(
\sqrt{\frac{1}{2}\,(\Sigmak{\iter-1})^2 (\constraint{\Sallk{\iter}}-\constraintmin)} -\norm{\Aall\multipliersol} 
\right)^2
-  \norm{\Aall\multipliersol}^2
 \quad \\\displaystyle \hfill
\geq
\frac{1}{2}\,(\Sigmak{\iter-1})^2 (\constraint{\Sallk{\iter}}-\constraintmin )   
-  \norm{\Aall\multipliersol}^2
\geq 
-   \norm{\Aall\multipliersol}^2
.
   \end{array}
\end{equation}
It follows that for all~$\xall$ the sequence $\lyapunovk{\iter}{\xall}\geq -\norm{\Aall\multipliersol}^2$ is bounded from below, and the supermartingale theorem applies in~\eqref{earlystronglyconvexfejersimple}, 
so that the sequence~$(\lyapunovk{\iter}{\xallsol} )$ converges almost surely to a random quantity not smaller than  $-\norm{\Aall\multipliersol}^2$, and
\begin{equation}
 \sum_{\iter=0}^{\infty} 
 \left[
    \frac{\dualstepsizek{\iter}}{2\stepsizek{\iter}} \normx{\hatxallk{\iter+1}-\xallk{\iter}}{\Dimension\Stepsize-\stepsizek{\iter}(\dualstepsizek{\iter}\constraintcovariance+\NSmoothness)}^2    
+
(\dualstepsizek{\iter})^2 (\constraint{\xallk{\iter}}-\constraintmin)
 \right]
 <
 \infty
 \quad
 \text{a.s.}
 .
\end{equation}
The convergence analysis of the sequences~$\xallk{\iter}$, $\Sallk{\iter}$ and~$\Zallk{\iter}$ is deferred to the proofs of Theorems~\ref{theorem:convergence}\eqref{theorem:convergence:i} and~\ref{theorem:newscconvergence}\eqref{theorem:newscconvergence:i} in the Appendix.

\subsection{Without a Lagrange multiplier}\label{section:withoutmultiplier}
Suppose now that no such optimal primal-dual pair exists. If~$\compositefunction$ is bounded from below by a finite constant~$\compositemin$, then $\hatcompositek{\iter}\geq\composite{\Sallk{\iter}}\geq \compositemin$,
and we find
\begin{equation}\label{boundkpenalizedkSCwithout}
 \penalizedk{\iter} 
 =
  \hatcompositek{\iter} -\compositeopt
 + \Sigmak{\iter-1} (\constraint{\Sallk{\iter}}-\constraintmin )  
%
\geq
\compositemin -\compositeopt
 + \Sigmak{\iter-1} (\constraint{\Sallk{\iter}}-\constraintmin )  
\geq
\compositemin -\compositeopt
.
\end{equation}
The supermartingale theorem can now be applied to the sequence defined by 
$\scaledlyapunovk{0}{\xallsol}=\lyapunovk{0}{\xallsol}=\frac{\dualstepsizek{0}}{2\stepsizek{0}} \normx{\xallk{0}-\xallsol}{\Dimension^2\Stepsize+\stepsizek{0}(\Dimension-\identityI{\primaldimension})\Convex}^2$
and 
$\scaledlyapunovk{\iter}{\xallsol}=\lyapunovk{\iter}{\xallsol}/\Sigmak{\iter-1}$ for $\iter\geq 1$. 
This downscaled sequence is bounded from below by $\compositemin-\compositeopt$ and satisfies
\begin{equation}  \label{earlystronglyconvexfejernobound}
\nocolsep
\begin{array}{l}  \displaystyle
\Condexpectation{\filtrationk{\iter}}{\scaledlyapunovk{\iter+1}{\xallsol} 
  }
  \leq
\frac{1}{\Sigmak{\iter-1}} \,\Condexpectation{\filtrationk{\iter}}{\lyapunovk{\iter+1}{\xallsol} 
  }
  \\  \qquad \displaystyle
\refereq{\eqref{earlystronglyconvexfejersimple}}{\leq}
\scaledlyapunovk{\iter}{\xallsol}
  - \frac{\dualstepsizek{\iter}}{2\stepsizek{\iter}\Sigmak{\iter-1}} \normx{\hatxallk{\iter+1}-\xallk{\iter}}{\Dimension\Stepsize-\stepsizek{\iter}(\dualstepsizek{\iter}\constraintcovariance+\NSmoothness)}^2  
-\frac{(\dualstepsizek{\iter})^2}{\Sigmak{\iter-1}} (\constraint{\xallk{\iter}}-\constraintmin)
,
\end{array}
\end{equation} 
where we have used $\Sigmak{\iter}\geq\Sigmak{\iter-1}$. 
We find that $\scaledlyapunovk{\iter}{\xallsol}$ converges almost surely to a random quantity $Z(\sample)\geq \compositemin-\compositeopt$, and 
\begin{equation} \label{doob}
 \sum_{\iter=0}^{\infty} 
 \left[
    \frac{\dualstepsizek{\iter}}{2\stepsizek{\iter}\Sigmak{\iter-1}} \normx{\hatxallk{\iter+1}-\xallk{\iter}}{\Dimension\Stepsize-\stepsizek{\iter}(\dualstepsizek{\iter}\constraintcovariance+\NSmoothness)}^2    
+
\frac{(\dualstepsizek{\iter})^2}{\Sigmak{\iter-1}} (\constraint{\xallk{\iter}}-\constraintmin)
 \right]
 <
 \infty
 \quad
 \text{a.s.}
 .
\end{equation}
See the proofs of Theorems~\ref{theorem:convergence}\eqref{theorem:convergence:ii} and~\ref{theorem:newscconvergence}\eqref{theorem:newscconvergence:ii} in the Appendix for a derivation of the convergence rates of the individual sequences. We note that the orders of convergence lost in the absence of a Lagrange multiplier are due to the scaling factor~$1/\Sigmak{\iter-1}$ in the sequence~$\scaledlyapunovk{\iter}{\xallsol}$.

\section{Decreasing stepsizes for the strongly convex case}\label{section:newstronglyconvexanalysis}

In this section we derive the decreasing stepsize sequence~\eqref{stepsizesequence} for the strongly convex case ($\Convex\succ 0$). First, we ensure~\eqref{stepsizeconditionone} holds by setting $
\dualstepsizek{\iter} = {\alpha}/{\stepsizek{\iter}} - \fixed{\betaparameter}
$, with $\alpha=\invspectralradius{\constraintcovariance(\Dimension\Stepsize)^{-1}}$ and 
$
\betaparameter
=
\spectralradius{\NSmoothness(\Dimension\Stepsize)^{-1}}\, \alpha $. 
%
%
Indeed,
we find
$
1/\stepsizek{\iter}
=
\dualstepsizek{\iter}/\alpha+
\spectralradius{\NSmoothness(\Dimension\Stepsize)^{-1}} 
=
\dualstepsizek{\iter}\,\spectralradius{\constraintcovariance(\Dimension\Stepsize)^{-1}}+
\spectralradius{\NSmoothness(\Dimension\Stepsize)^{-1}} 
$. 
Consequently,
\begin{equation*}
\begin{array}{l}
\displaystyle
\stepsizek{\iter}(\dualstepsizek{\iter}\constraintcovariance+\NSmoothness) 
- 
\Dimension\Stepsize
=
\stepsizek{\iter}\left((\dualstepsizek{\iter}\constraintcovariance+\NSmoothness) 
(\Dimension\Stepsize)^{-1}
- 
\frac{\identityI{\primaldimension}}{\stepsizek{\iter}}
\right)
\Dimension\Stepsize
\quad\\\qquad\hfill\displaystyle
=
\stepsizek{\iter}\left((\dualstepsizek{\iter}\constraintcovariance+\NSmoothness) 
(\Dimension\Stepsize)^{-1}
- 
[ \dualstepsizek{\iter}\,\spectralradius{\constraintcovariance(\Dimension\Stepsize)^{-1}}+
\spectralradius{\NSmoothness(\Dimension\Stepsize)^{-1}} ] \identityI{\primaldimension}
\right)
\Dimension\Stepsize
 \preceq 
0
\, ,
\end{array}
\end{equation*}
and~\eqref{stepsizeconditionone} is true. 
%
%
%
%
Then, solving~\eqref{stepsizecondition} for~$\stepsizek{\iter}$ yields a lower bound for the sequence~$\stepsizek{\iter}$, which must satisfy
%
%
\begin{equation*}  \label{newstepsizelowerboundcondition}
\begin{array}{l}
 \lbrack 
 (\alpha - \fixed{\betaparameter} \stepsizek{\iter}) (\Dimension^2\Stepsize +\stepsizek{\iter}\nomoreconvexratio\Dimension\Convex ) 
 -  \fixed{\betaparameter} (\stepsizek{\iter})^2 \fixed{(\identityI{\primaldimension}-\Dimension)}\Convex \rbrack
 (\stepsizek{\iter+1})^2 
 + (\stepsizek{\iter})^2 (\fixed{\betaparameter}  \Dimension^2\Stepsize+  \alpha \fixed{(\identityI{\primaldimension}-\Dimension)}\Convex) \stepsizek{\iter+1} 
 \quad\\\hfill
 - \alpha (\stepsizek{\iter})^2  \Dimension^2\Stepsize 
\succeq 0
.
\end{array}
\end{equation*}
A close inspection (omitted in this study) of the above condition yields an upper bound for the sequence~$(\stepsizek{\iter})$, which decreases with asymptotic convergence  rate 
$\magnitude{2\spectralradius{\Convex^{-1}\Dimension^2\Stepsize}/\iter}$. 
In view of this conjecture, we can maximize the sequence by balancing the eigenspectrum of the matrix~$\Convex^{-1}\Dimension^2\Stepsize$. We do so by choosing~%
$\Stepsize=\Diagonal\Convex$, where~$\Diagonal=\diag(\diagonali{1}\identityI{\primaldimensioni{1}},\dots,\diagonali{\dimension}\identityI{\primaldimensioni{\dimension}})$, and the condition 
becomes
\begin{equation} \label{stepsizelowerboundconditionconjecturedmatrix}
\begin{array}{l}
 \big[\lbrack 
 (\alpha - \fixed{\betaparameter} \stepsizek{\iter}) (\Dimension^2\Diagonal +\stepsizek{\iter}\nomoreconvexratio\Dimension ) 
 -  \fixed{\betaparameter} (\stepsizek{\iter})^2 \fixed{(\identityI{\primaldimension}-\Dimension)} \rbrack
 (\stepsizek{\iter+1})^2 
 + (\stepsizek{\iter})^2 (\fixed{\betaparameter}  \Dimension^2\Diagonal +  \alpha (\identityI{\primaldimension}-\Dimension)) \stepsizek{\iter+1} 
\quad\\\hfill
 - \alpha (\stepsizek{\iter})^2  \Dimension^2\Diagonal
\big] \Convex \succeq 0
.
\end{array}
\end{equation}
where $\alpha=\invspectralradius{\constraintcovariance(\Dimension\Diagonal\Convex)^{-1}}$ and 
$
\fixed{\betaparameter}
=
\spectralradius{\NSmoothness(\Dimension\Diagonal\Convex)^{-1}}\, \alpha
$.
The stepsize sequence~\eqref{stepsizesequence} is obtained by letting $\Diagonal=\Dimension^{-2}$ (i.e., $\Stepsize=\Dimension^{-2}\Convex$) and by recursively taking for~$\stepsizek{\iter+1}$ the positive solution of~\eqref{stepsizelowerboundconditionconjecturedmatrix}. 
As we proceed to show next, the sequence~\eqref{stepsizesequence} is decreasing with asymptotic rate~${2/\iter}$.

\smallskip

Fix~$\Stepsize=\Diagonal\Convex$, where~$\Diagonal=\diag(\diagonali{1}\identityI{\primaldimensioni{1}},\dots,\diagonali{\dimension}\identityI{\primaldimensioni{\dimension}})$, and ensure that~\eqref{stepsizeconditionone} holds by setting $
\dualstepsizek{\iter} = {\alpha}/{\stepsizek{\iter}} - \fixed{\betaparameter}
$, where
$\alpha=\invspectralradius{\constraintcovariance(\Dimension\Diagonal\Convex)^{-1}}$ and 
$
\fixed{\betaparameter}
=
\spectralradius{\NSmoothness(\Dimension\Diagonal\Convex)^{-1}}\, \alpha
$. Then, Condition~\eqref{stepsizeconditiontwo} reduces to~\eqref{stepsizelowerboundconditionconjecturedmatrix}, which rewrites as

\begin{equation} \label{stepsizelowerboundconditionconjectured}
\begin{array}{l}
\lbrack 
 (\alpha - \betaparameter \stepsizek{\iter}) (\diagonali{\coordinate} + \probai{\coordinate}\stepsizek{\iter} )  
 +  \betaparameter\probai{\coordinate}  (1-\probai{\coordinate}) (\stepsizek{\iter})^2 \rbrack
 (\stepsizek{\iter+1})^2 
\qquad\qquad\qquad\qquad\qquad\\\hfill 
 +  (\betaparameter  \diagonali{\coordinate} -  \alpha \probai{\coordinate}(1-\probai{\coordinate})) (\stepsizek{\iter})^2 \stepsizek{\iter+1} - \alpha \diagonali{\coordinate} (\stepsizek{\iter})^2  
\geq 0, 
\quad \coordinate=1,\dots,\dimension.
\end{array}
\end{equation}
The discriminant of~\eqref{stepsizelowerboundconditionconjectured}, given by
\begin{equation}
 \label{discriminant}
\begin{aligned}
 \Delta
 &= 
  4 \alpha^2  \left\lbrack  1 + \left( \frac{ \probai{\coordinate}}{ \diagonali{\coordinate}}  - \frac{\betaparameter }{\alpha } \right)\stepsizek{\iter}+ \left[ \left( \frac{ \probai{\coordinate}}{ \diagonali{\coordinate}}  - \frac{\betaparameter }{\alpha } \right)^2  -        \left( \frac{ \probai{\coordinate}(2-\probai{\coordinate})}{\diagonali{\coordinate}}  +   \frac{2\betaparameter}{\alpha}    \right) \frac{\probai{\coordinate}^2}{\diagonali{\coordinate}} \right] \frac{(\stepsizek{\iter})^2}{4}   \right\rbrack  \diagonali{\coordinate}^2 (\stepsizek{\iter})^2 
\\
&=
  4 \alpha^2  \left\lbrack  \left(1 +  \left( \frac{ \probai{\coordinate}}{ \diagonali{\coordinate}}  - \frac{\betaparameter }{\alpha } \right)\frac{\stepsizek{\iter}}{2}\right)^2 -      \left( \frac{ \probai{\coordinate}(2-\probai{\coordinate})}{\diagonali{\coordinate}}  +   \frac{2\betaparameter}{\alpha}    \right) \frac{\probai{\coordinate}^2(\stepsizek{\iter})^2}{4\diagonali{\coordinate}}   \right\rbrack  \diagonali{\coordinate}^2 (\stepsizek{\iter})^2 
,
\end{aligned}
\end{equation}
is nonnegative for
\begin{condition}
  \label{discriminantnonnegative}
\stepsizek{\iter}
\leq   
\frac{2}{
\max\Big(0, \big( \frac{\betaparameter }{\alpha } - \frac{ \probai{\coordinate}}{ \diagonali{\coordinate}} \big)
-
\frac{\probai{\coordinate}}{\sqrt{\diagonali{\coordinate}}} 
\sqrt{ \frac{ \probai{\coordinate}(2-\probai{\coordinate})}{\diagonali{\coordinate}}  +   \frac{2\betaparameter}{\alpha} } 
,
\frac{\probai{\coordinate}}{\sqrt{\diagonali{\coordinate}}} 
\sqrt{ \frac{ \probai{\coordinate}(2-\probai{\coordinate})}{\diagonali{\coordinate}}  +   \frac{2\betaparameter}{\alpha} } \,  -  \big( \frac{ \probai{\coordinate}}{ \diagonali{\coordinate}}  - \frac{\betaparameter }{\alpha } \big)\Big)
}
.
\end{condition}
Under~\eqref{discriminantnonnegative},
Condition~\eqref{stepsizelowerboundconditionconjectured} reduces to
\begin{align}
\stepsizek{\iter+1}
&\geq
\frac{- (\betaparameter  \diagonali{\coordinate} -  \alpha \probai{\coordinate}(1-\probai{\coordinate})) (\stepsizek{\iter})^2  +\sqrt{\Delta}}{2 \lbrack 
 (\alpha - \betaparameter \stepsizek{\iter}) (\diagonali{\coordinate} + \probai{\coordinate}\stepsizek{\iter} )  
 +  \betaparameter\probai{\coordinate}  (1-\probai{\coordinate}) (\stepsizek{\iter})^2 \rbrack}
 \label{rateearlyresult} 
 ,
 \end{align} 
holding for $\coordinate=1,\dots,\dimension$.
\obsolete{So we can take the largest root for the sequence (using $=$ sign).}%
Now, define~$\differencemin$, $\differencemax$ and~$\difference$ as
\begin{equation}\label{difference}
 \differencemin = \min_{\coordinate\in\{1,\dots,\dimension\}} \left\{ \frac{\probai{\coordinate}}{\diagonali{\coordinate}}-\frac{\betaparameter }{\alpha} \right\},
 \quad
 \differencemax = \max_{\coordinate\in\{1,\dots,\dimension\}} \left\{ \frac{\probai{\coordinate}}{\diagonali{\coordinate}}-\frac{\betaparameter }{\alpha} \right\},
 \quad
 \difference=\max(\modulus{\differencemin},\modulus{\differencemax})
.
\end{equation}
Let $\cparameter>1$, and assume that
\begin{equation} \label{newconditionstepsizeone}
-\frac{\cparameter-1}{\cparameter}  \leq   \left(\frac{\betaparameter\probai{\coordinate}^2}{\alpha \diagonali{\coordinate}}\right) (\stepsizek{\iter})^2 -\left( \frac{\probai{\coordinate}}{\diagonali{\coordinate}}-\frac{\betaparameter}{\alpha} \right)\stepsizek{\iter}  \leq \frac{\cparameter-1}{\cparameter} 
,
 \end{equation}
or, equivalently,
\begin{condition} \label{newconditionstepsizeoneexplicit}
\begin{aligned}
&
\stepsizek{\iter}
\leq
\left(\frac{\alpha \diagonali{\coordinate}}{2\betaparameter\probai{\coordinate}^2}\right) \left[{\differencemin+\sqrt{\differencemin^2+ 4 \left(\frac{\betaparameter\probai{\coordinate}^2}{\alpha \diagonali{\coordinate}}\right)\left(\frac{\cparameter-1}{\cparameter}\right)}}\right]
, \qquad
\\ &
\stepsizek{\iter}
\leq
\left(\frac{\alpha \diagonali{\coordinate}}{2\betaparameter\probai{\coordinate}^2}\right)\left[\differencemax-\sqrt{\differencemax^2- 4 \left(\frac{\betaparameter\probai{\coordinate}^2}{\alpha \diagonali{\coordinate}}\right)\left(\frac{\cparameter-1}{\cparameter}\right)}\right]
\text{ if }\differencemax>0.
\end{aligned}
\end{condition}
 Using the identities $\sqrt{1+x}\leq 1 + (x/2)$, and $(1-x)^{-1}\leq 1+x+\cparameter x^2$ for $x\leq 1-1/\cparameter $, we find
%
\allowdisplaybreaks
\begin{align}
 \nonumber
 &
\frac{1}{\stepsizek{\iter}} \left(
\frac{- (\betaparameter  \diagonali{\coordinate} -  \alpha \probai{\coordinate}(1-\probai{\coordinate})) (\stepsizek{\iter})^2  +\sqrt{\Delta}}{2 \lbrack 
 (\alpha - \betaparameter \stepsizek{\iter}) (\diagonali{\coordinate} + \probai{\coordinate}\stepsizek{\iter} )  
 +  \betaparameter\probai{\coordinate}  (1-\probai{\coordinate}) (\stepsizek{\iter})^2 \rbrack}
 \right)
\\
 \nonumber & \quad
\refereq{\eqref{discriminant}}{\leq}
\frac{ 1 + (\frac{ \probai{\coordinate}}{ \diagonali{\coordinate}}-\frac{\betaparameter}{\alpha}-\frac{\probai{\coordinate}^2}{ 2\diagonali{\coordinate}}  )  \stepsizek{\iter} + \left[ \left( \frac{ \probai{\coordinate}}{ \diagonali{\coordinate}}  - \frac{\betaparameter }{\alpha } \right)^2
-  \left( \frac{ \probai{\coordinate}(2-\probai{\coordinate})}{\diagonali{\coordinate}}  +   \frac{2\betaparameter}{\alpha}    \right) \frac{\probai{\coordinate}^2}{\diagonali{\coordinate}} \right]  \frac{(\stepsizek{\iter})^2}{8}  
}{
   1 -\big[ (\frac{\betaparameter\probai{\coordinate}^2}{\alpha \diagonali{\coordinate}}) (\stepsizek{\iter})^2 - (\frac{\probai{\coordinate}}{\diagonali{\coordinate}}-\frac{\betaparameter }{\alpha})\stepsizek{\iter}\big] } 
\\
 \nonumber & \quad \textstyle
\leq
\left[ 1 + \left(\frac{ \probai{\coordinate}}{ \diagonali{\coordinate}}-\frac{\betaparameter}{\alpha}-\frac{\probai{\coordinate}^2}{ 2\diagonali{\coordinate}}  \right)  \stepsizek{\iter} + 
\left[ \left( \frac{ \probai{\coordinate}}{ \diagonali{\coordinate}}  - \frac{\betaparameter }{\alpha } \right)^2
-  \left( \frac{ \probai{\coordinate}(2-\probai{\coordinate})}{\diagonali{\coordinate}}  +   \frac{2\betaparameter}{\alpha}    \right) \frac{\probai{\coordinate}^2}{\diagonali{\coordinate}} \right]  \frac{(\stepsizek{\iter})^2}{8}  
\right]
\\ \nonumber &   \pushright{ \textstyle \times
\left[
   1 + \left[ \left(\frac{\betaparameter\probai{\coordinate}^2}{\alpha \diagonali{\coordinate}}\right) (\stepsizek{\iter})^2 
- 
   \left(\frac{\probai{\coordinate}}{\diagonali{\coordinate}}-\frac{\betaparameter }{\alpha}\right)\stepsizek{\iter}
   \right] + \cparameter \left[\left(\frac{\probai{\coordinate}}{\diagonali{\coordinate}}-\frac{\betaparameter }{\alpha}\right)\stepsizek{\iter}  
 - \left(\frac{\betaparameter\probai{\coordinate}^2}{\alpha \diagonali{\coordinate}}\right) (\stepsizek{\iter})^2 \right]^2  \right] 
 }
\\
 \nonumber & \quad\textstyle =
1
-
\left(\frac{\probai{\coordinate}^2}{ 2\diagonali{\coordinate}}   \right) \stepsizek{\iter}
+
\left(\frac{\betaparameter\probai{\coordinate}^2}{\alpha \diagonali{\coordinate}}\right) (\stepsizek{\iter})^2
+
\cparameter \left[ \left(\frac{\betaparameter\probai{\coordinate}^2}{\alpha \diagonali{\coordinate}}\right) (\stepsizek{\iter})^2 - \left(\frac{\probai{\coordinate}}{\diagonali{\coordinate}}-\frac{\betaparameter }{\alpha}\right)\stepsizek{\iter}\right]^2
\\ \nonumber & \quad \textstyle \quad
+ \left[ \left( \frac{ \probai{\coordinate}}{ \diagonali{\coordinate}}  - \frac{\betaparameter }{\alpha } \right)^2
-  \left( \frac{ \probai{\coordinate}(2-\probai{\coordinate})}{\diagonali{\coordinate}}  +   \frac{2\betaparameter}{\alpha}    \right) \frac{\probai{\coordinate}^2}{\diagonali{\coordinate}} \right]  \frac{(\stepsizek{\iter})^2}{8} 
\\
 \nonumber & \quad\textstyle \quad +
\left[
\left[ \left(\frac{\betaparameter\probai{\coordinate}^2}{\alpha \diagonali{\coordinate}}\right) (\stepsizek{\iter})^2 
- 
   \left(\frac{\probai{\coordinate}}{\diagonali{\coordinate}}-\frac{\betaparameter }{\alpha}\right)\stepsizek{\iter}
   \right] 
+ 
\cparameter \left[\left(\frac{\probai{\coordinate}}{\diagonali{\coordinate}}-\frac{\betaparameter }{\alpha}\right)\stepsizek{\iter}  
 - \left(\frac{\betaparameter\probai{\coordinate}^2}{\alpha \diagonali{\coordinate}}\right) (\stepsizek{\iter})^2 \right]^2
\right]
\left(\frac{ \probai{\coordinate}}{ \diagonali{\coordinate}}-\frac{\betaparameter}{\alpha}-\frac{\probai{\coordinate}^2}{ 2\diagonali{\coordinate}}  \right)  \stepsizek{\iter} 
\\
 \nonumber & \quad\textstyle \quad +
\left[
\left[ \left(\frac{\betaparameter\probai{\coordinate}^2}{\alpha \diagonali{\coordinate}}\right) (\stepsizek{\iter})^2 
- 
   \left(\frac{\probai{\coordinate}}{\diagonali{\coordinate}}-\frac{\betaparameter }{\alpha}\right)\stepsizek{\iter}
   \right] 
+ 
\cparameter \left[\left(\frac{\probai{\coordinate}}{\diagonali{\coordinate}}-\frac{\betaparameter }{\alpha}\right)\stepsizek{\iter}  
 - \left(\frac{\betaparameter\probai{\coordinate}^2}{\alpha \diagonali{\coordinate}}\right) (\stepsizek{\iter})^2 \right]^2
\right]
\\ \nonumber & \pushright{\textstyle \times
\left[ \left( \frac{ \probai{\coordinate}}{ \diagonali{\coordinate}}  - \frac{\betaparameter }{\alpha } \right)^2
-  \left( \frac{ \probai{\coordinate}(2-\probai{\coordinate})}{\diagonali{\coordinate}}  +   \frac{2\betaparameter}{\alpha}    \right) \frac{\probai{\coordinate}^2}{\diagonali{\coordinate}} \right]  \frac{(\stepsizek{\iter})^2}{8}  
}
\\
 \nonumber &  \refereq{\eqref{difference},\eqref{newconditionstepsizeone}}{\leq} \textstyle
1
-
\left(\frac{\probai{\coordinate}^2}{ 2\diagonali{\coordinate}}   \right) \stepsizek{\iter}
+
\cparameter \left[ \frac{17\difference^2}{8} + \left( \frac{\difference}{2} + \frac{\betaparameter}{\alpha} \right)   \frac{\probai{\coordinate}^2}{\diagonali{\coordinate}} + \frac{1+2\difference}{8} \left(\frac{\probai{\coordinate}^2}{\diagonali{\coordinate}}\right)^2  \right] 
(\stepsizek{\iter})^2 
 + \cparameter \left(2 \difference + \frac{\probai{\coordinate}^2}{ 2\diagonali{\coordinate}}  \right)   \left(\frac{\betaparameter\probai{\coordinate}^2}{\alpha \diagonali{\coordinate}}\right) (\stepsizek{\iter})^3 
\\   
 & \quad 
=
1
-
\left(\frac{\probai{\coordinate}^2}{ 2\diagonali{\coordinate}}   \right) \stepsizek{\iter}
+
\crap{\stepsizek{\iter}}
,
\label{ratefinalresult}
\end{align}
\interdisplaylinepenalty=10000
where $\crap{\epsilon}=\magnitude{\epsilon^2}
$. %
%
The above result can be balanced by setting $\diagonali{\coordinate} = \probai{\coordinate}^2$ for $\coordinate=1,\dots,\dimension$.
In particular, it can be seen from~\eqref{rateearlyresult} and~\eqref{ratefinalresult} that $\stepsizek{\iter+1}
\leq 
 \stepsizek{\iter}
-
(1-\epsilon) (\stepsizek{\iter})^2/{2}$ whenever
\begin{condition}\label{newconditionstepsizetwo}
\stepsizek{\iter}
\leq
\left[\frac{ (17\difference^2+ 6\difference+ 1)\alpha + 8\betaparameter}{8 \left(4 \difference + 1  \right)   \betaparameter }\right]
\left[\sqrt{1+\frac{16 \alpha \betaparameter(4 \difference + 1)\, \epsilon   }{ [(17\difference^2+ 6\difference+ 1)\alpha + 8\betaparameter ]^2 \cparameter } }-1\right]
%
.
\end{condition}
It follows that the stepsize sequence~\eqref{stepsizesequencepolynomial}
satisfies~\eqref{stepsizecondition} and decreases with complexity
$
\stepsizek{\iter+1}
\leq
 \stepsizek{\iter}
-
 ({1-\epsilon})(\stepsizek{\iter})^2/2
$ 
on condition that~$\stepsizek{0}$ satisfies~\eqref{discriminantnonnegative}, \eqref{newconditionstepsizeoneexplicit} and~\eqref{newconditionstepsizetwo}. 
 
Now we show that the tight stepsize rule~\eqref{stepsizesequence}
vanishes with asymptotic behavior 
$
\stepsizek{\iter+1}
=
 \stepsizek{\iter}
-
{(\stepsizek{\iter})^2}/{2}
+
\magnitude{(\stepsizek{\iter})^3}
$. 
To see this, observe that~\eqref{stepsizelowerboundconditionconjectured} has two fixed points, $0$ and $\frac{\alpha}{\betaparameter}$, where~$0$  is the unique attracting fixed point. 
Indeed, the positive solution of~\eqref{stepsizelowerboundconditionconjectured}  with $\diagonali{\coordinate}= \probai{\coordinate}^2$ is  given by~\eqref{stepsizesequence}, which reduces for $\stepsizek{\iter}=\frac{\alpha}{\betaparameter}-\epsilon$ to
\begin{equation}
 \stepsizek{\iter+1}
=
\stepsizek{\iter} -\frac{\betaparameter\probai{\coordinate}+\alpha}{\betaparameter\probai{\coordinate}+(1-\probai{\coordinate}) \alpha} \epsilon +\magnitude{\epsilon^2} 
.
\end{equation}
Hence, stepsize sequence~\eqref{stepsizesequence} will run away from~$\alpha/\betaparameter$ and  decrease towards~$0$ provided that $\stepsizek{0}<\frac{\alpha}{\betaparameter}$. Moreover, we will have $
\dualstepsizek{\iter} > 0
$ for all~$\iter$ in~\eqref{newdualstepsizek}.

To find the asymptotic rate of~$\stepsizek{\iter}$ under $\diagonali{\coordinate} = \probai{\coordinate}^2$ for $\coordinate=1,\dots,\dimension$, we apply $\sqrt{1+x}= 1 + (x/2)+ \magnitude{x^2}$ and $(1-x)^{-1}\leq 1+x+ \magnitude{x^2}$ to~\eqref{stepsizesequence}. Proceeding as in~\eqref{ratefinalresult}, we find $\stepsizek{\iter+1}=\stepsizek{\iter}-(\stepsizek{\iter})^2/2+\magnitude{(\stepsizek{\iter})^3}$. 
We infer that $\stepsizek{\iter}-2/\iter=\zero{1/\iter}$, and it follows from~\eqref{newdualstepsizek} that $\dualstepsizek{\iter} - ({\alpha\iter}/{2} - \fixed{\betaparameter} ) = \zero{1} $, $ \Sigmak{\iter}- [ \alpha\iter(\iter+1)/{4} - (\iter+1) \fixed{\betaparameter}] = \zero{\iter} $.\footnotemark{} 
\footnotetext{If $0{\,<\,}a[1]{\,<\,}2$, the solution of the difference equation $a[\iter+1]{\,=\,}a[\iter]{\,-\,}\frac 1 2 a[\iter]^2$ is $a[\iter]{\,=\,}\frac 2 \iter{\,+\,}\zero{1/\iter}$ for $\iter{\,\geq\,} 1$.
}
%
%


\section{Numerical experiments for an optimal transport problem}\label{section:optimaltransport}

\newcommand{\newcongestion}[2]{#1}%
\newcommand{\monopolistic}[2]{#1}%

\renewcommand{\indicatorfunction}[1]{\mathbb{I}_{#1}}%

\newcommand{\duali}[1]{\dualvariable_{#1}}%

\newcommand{\OTxset}{X}%
\newcommand{\OTxsetj}[1]{\OTxset_{#1}}%
\newcommand{\OTxsetsol}{\OTxset^*}%
\newcommand{\OTx}{x}%
\newcommand{\OTxj}[1]{\OTx_{#1}}%
\newcommand{\OTxij}[2]{\OTx_{{#1}{#2}}}%
\newcommand{\OTxk}[1]{\OTx^{#1}}%
\newcommand{\OTxkj}[2]{\OTxk{#1}_{#2}}%
\newcommand{\OTxkij}[3]{\OTxk{#1}_{{#2}{#3}}}%
\newcommand{\altOTx}{\tilde{\OTx}}%
\newcommand{\altOTxj}[1]{\altOTx_{#1}}%
\newcommand{\altOTxij}[2]{\altOTx_{{#1}{#2}}}%
\newcommand{\altOTxk}[1]{\altOTx^{#1}}%
\newcommand{\altOTxkj}[2]{\altOTxk{#1}_{#2}}%
\newcommand{\altOTxkij}[3]{\altOTxk{#1}_{{#2}{#3}}}%
\newcommand{\OTz}{z}%
\newcommand{\OTzj}[1]{\OTz_{#1}}%
\newcommand{\identityj}[1]{\identity_{#1}}%

\newcommand{\constant}{a}%
\newcommand{\costmaxi}[1]{\bar{\costsymbol}_{#1}}%
\newcommand{\strongconvexitymodulus}{M}%
\newcommand{\strongconvexitymodulusj}[1]{\strongconvexitymodulus_{#1}}%
\newcommand{\strongconvexitymodulusij}[2]{\strongconvexitymodulus_{{#1}{#2}}}%

\newcommand{\price}{p}%
\newcommand{\pricej}[1]{\price_{#1}}%
\newcommand{\revenuefunction}{R}%
\newcommand{\revenue}[1]{\revenuefunction({#1})}%

\newcommand{\xmeasurefunction}{\mu}%
\newcommand{\xmeasure}[1]{\xmeasurefunction(#1)}%
\newcommand{\xMeasure}[1]{\xmeasurefunction\left(#1\right)}%
\newcommand{\ymeasurefunction}{\nu}%
\newcommand{\ymeasure}[1]{\ymeasurefunction(#1)}%
\newcommand{\ymeasurei}[1]{\ymeasurefunction_{#1}}%
\newcommand{\ymeasurevec}{\bm{\ymeasurefunction}}%
\newcommand{\ymeasureveci}[1]{\ymeasurevec_{#1}}%
\newcommand{\xmeasurei}[1]{\xmeasurefunction_{#1}}%
\newcommand{\xmeasurevec}{\bm{\xmeasurefunction}}%
\renewcommand{\xmeasurevec}{\xmeasurefunction}%
\newcommand{\altxmeasurefunction}{\tilde{\xmeasurefunction}}%
\newcommand{\altxmeasure}[1]{\altxmeasurefunction(#1)}%
\newcommand{\altymeasurefunction}{\tilde{\ymeasurefunction}}%
\newcommand{\altymeasure}[1]{\altymeasurefunction(#1)}%
\newcommand{\nx}{m}%
\newcommand{\ny}{n}%
\newcommand{\costsymbol}{c}%
\newcommand{\costj}[1]{\costsymbol_{#1}}%
\newcommand{\costij}[2]{\costsymbol_{{#1}{#2}}}%
\newcommand{\congestionfunction}{g}%
\newcommand{\congestion}[1]{\congestionfunction({#1})}%
\newcommand{\congestionjfunction}[1]{\congestionfunction_{#1}}%
\newcommand{\congestionj}[2]{\congestionjfunction{#1}({#2})}%
\newcommand{\smoothjfunction}[1]{\smoothfunction_{#1}}%
\newcommand{\smoothj}[2]{\smoothjfunction{#1}({#2})}%
\newcommand{\nonsmoothjfunction}[1]{\nonsmoothfunction_{#1}}%
\newcommand{\nonsmoothj}[2]{\nonsmoothjfunction{#1}({#2})}%
\newcommand{\smoothcongestionfunction}{d}%
\newcommand{\smoothcongestion}[1]{\smoothcongestionfunction({#1})}%
\newcommand{\smoothcongestionjfunction}[1]{\smoothcongestionfunction_{#1}}%
\newcommand{\smoothcongestionj}[2]{\smoothcongestionjfunction{#1}({#2})}%
\newcommand{\dsmoothcongestionfunction}{\smoothcongestionfunction^\prime}%
\newcommand{\dsmoothcongestion}[1]{\dsmoothcongestionfunction({#1})}%
\newcommand{\dsmoothcongestionjfunction}[1]{\dsmoothcongestionfunction_{#1}}%
\newcommand{\dsmoothcongestionj}[2]{\dsmoothcongestionjfunction{#1}({#2})}%

\renewcommand{\ny}{\dimension}%

We consider an optimal transport problem where a number of customers are willing to buy a service available at various sites $1,\dots,\ny$. 
The customers are identical within each of~$\nx$ distinct classes indexed by $1,\dots,\nx$. Site~$j$ is characterized by a cost vector $\costj{j}=(\costij{1}{j},\dots,\costij{\nx}{j})$, where~$\costij{i}{j}$ denotes the cost charged to any customer from class~$i$ for buying the service at site~$j$ ($i=1,\dots,\nx$; $j=1,\dots,\ny$). 
We use the following notation: for $i=1,\dots,\nx$ and $j=1,\dots,\ny$,
\renewcommand{\labelitemi}{$\bullet$}%
\begin{itemize}

\item
$\xmeasurei{i} \in \Realpluszero$: mass of customers from class~$i$, 

\item
$\ymeasurei{j} \in \Realpluszero$: capacity at site~$j$, 

\item
$\OTxij{i}{j} \in \Realpluszero$: mass of customers from class~$i$ receiving service at site~$j$, 

\item
$\OTxj{j}=(\OTxij{1}{j},\dots,\OTxij{\nx}{j}) \in \Realpluszero^\nx$: service schedule at site~$j$, 


\item
$\OTx=(\OTxj{1},\dots,\OTxj{\ny}) \in \Realpluszero^{\nx\ny}$:
global service schedule, 


\item 
$\costj{j}=(\costij{1}{j},\dots,\costij{\nx}{j})$: cost vector associated with site~$j$,

\item 
$\costij{i}{j}$: cost charged to a customer from class~$i$ when receiving service at site~$j$.
\end{itemize}
\noindent
One would like to the maximize the overall attractiveness of the service by solving
\begin{equation}\label{problem:scheduling}\tag{\ensuremath{\text{OT}}}
 \begin{array}{ccl}
  \underset{\OTx}{\text{minimize }}
  &
  \sum_{j=1}^{\ny} \left[
  \inner{\costj{j}}{\OTxj{j}} + \smoothcongestionj{j}{\OTxj{j}}\right]
&
\quad
\\
\text{subject to}
& 
  \sum_{j=1}^{\ny} \OTxij{i}{j} = \xmeasurei{i}
&
(i=1,\dots,\nx)
\\
& 
  \sum_{i=1}^{\nx} \OTxij{i}{j}
\leq  \ymeasurei{j}
&
(j=1,\dots,\ny)
\\
  & 
  \OTxij{i}{j} \geq 0 
&
(i=1,\dots,\nx,\, j=1,\dots,\ny)
 \end{array}
\end{equation}
where~$ \smoothcongestionjfunction{j} : \Realpluszero^\nx \mapsto \Realpluszero$ is a smooth, convex function 
such that~$\smoothcongestionj{j}{\OTxj{j}}$ models an additional congestion cost around site~$j$ under service schedule~$\OTxj{j}$. 

Problem~\eqref{problem:scheduling} rewrites as as an instance of~\eqref{initialP} in which~$\primaldimensioni{j}\equiv \nx$ for $j=1,\dots,\dimension$,  the equality constraint is specified by $\bvector=\xmeasurevec=(\xmeasurei{1},\dots,\xmeasurei{\nx})$ and
$\Aall = (\AI{1}\ \cdots \ \AI{\ny}) $ with $\AI{j}=\identityj{\primaldimensioni{j}}$, the feasible sets are given for $j=1,\dots,\ny$ by
\begin{equation} 
 \OTxsetj{j}=\{\OTxj{j}\in\Realpluszero^\nx : \ymeasurei{j}-||\OTxj{j}||_1 \geq 0  
 \}
 ,
 \end{equation}
and the 
objective functions by
\begin{equation}
\smoothj{j}{\OTxj{j}} = \inner{\costj{j}}{\OTxj{j}} + \smoothcongestionj{j}{\OTxj{j}} -\frac{1}{2}\normx{\OTxj{j}}{\Convexi{j}}^2
,
\quad
 \nonsmoothj{j}{\OTxj{j}} 
 =  
 \indicator{\OTxsetj{j}}{\OTxj{j}} + \frac{1}{2}\normx{\OTxj{j}}{\Convexi{j}}^2
%
,
\end{equation}
where $\smoothjfunction{j}$ is smooth convex and $\nonsmoothjfunction{j}$ is proper convex lower semi-continuous, 
and~$\indicatorfunction{C} $ denotes the indicator function of the set~$C$.
%
%
%
The stationarity condition for~\eqref{problem:scheduling} implies that every primal-dual solution~$(\OTx,\dualvariable)$  satisfies
\begin{eqnarray}
\costij{i}{j}+
\newcongestion{
\gradI{j}\smoothcongestionj{j}{\OTxj{j}}
}{
\dsmoothcongestionj{j}{\OTxj{j}} 
}
+\duali{i} > 0
&
\Longrightarrow
&
\OTxij{i}{j}=0
,
\\
\OTxij{i}{j}>0
&
\Longrightarrow
&
\costij{i}{j}+
\newcongestion{
\gradI{j}\smoothcongestionj{j}{\OTxj{j}}
}{
\dsmoothcongestionj{j}{\OTxj{j}} 
}
+ \delta_j +\duali{i} = 0
,
\end{eqnarray}
for some $\delta_j\geq 0$, where~$\delta_j$ denotes the dual variable relative to the side inequality constraint $  \sum_{i=1}^{\nx} \OTxij{i}{j}
-  \ymeasurei{j}
\leq 0
$, and~$\duali{i}$ relates to the equality constraint $\sum_{j=1}^{\ny} \OTxij{i}{j} - \xmeasurei{i}=0$.

\subsection*{Customer guidance through pricing.}
The interpretation of these dual variables as prices allows the operator to maximize the service quality by
assigning unit prices $\pricej{1},\dots,\pricej{\ny}$ to the service at each site, under the assumption that all customers are willing to be served at lowest cost, i.e., a customer from class~$i$ is going to purchase the service at a site~$j$ that minimizes their prospective total cost $\costij{i}{j}+\pricej{j}$.
\monopolistic{}{
For this we may consider two settings:
\begin{enumerate}[(i)]
 \item 
 \label{contextmono}
 A monopolistic context, where the customers are committed to puchasing the sevice and the operator must secure a fixed revenue for the service provider, 
 \item 
 \label{contextfree}
 A free market environment, where service must be guaranteed under the additional constraint that each customer from a given class may reject the service if its minimal prospective cost exceeds a certain threshold.
\end{enumerate}
}%
Indeed, implementing the price profile 
\begin{equation}\label{pricingsolution}
 \pricej{j}=  
 \newcongestion{
\gradI{j}\smoothcongestionj{j}{\OTxj{j}}
}{
\dsmoothcongestionj{j}{\OTxj{j}} 
}
+ \delta_j - \constant 
 , 
\end{equation}
where~$\constant$ is an arbitrary constant parameter, yields
\begin{eqnarray}
 \OTxij{i}{j}>0 
 &
 \Longrightarrow
 &
 \costij{i}{j}+ \pricej{j}   +\duali{i} + \constant = 0
,
\\
\costij{i}{j} + \pricej{j} +\duali{i} + \constant > 0
&
\Longrightarrow
&
\OTxij{i}{j}=0
\end{eqnarray}
or, equivalently,
\begin{equation}\label{decision}
\OTxij{i}{j}>0 
 \Longrightarrow
j\in\arg\min\nolimits_{j^\prime} \{ \costij{i}{j^\prime}+ \pricej{j^\prime}  \}
.
\end{equation}
which is the desired result. Any indecisions in~\eqref{decision} may be dealt with in practice using a booking system.

\subsection*{Experiments.}
\newcommand{\proj}[2]{[{#1}]^+{#2}}%
\newcommand{\Proj}[2]{\left[{#1}\right]^+{#2}}%
\newcommand{\bigproj}[2]{\big[{#1}\big]^+{#2}}%
\renewcommand{\scalingsymbol}{\lambda}%
\renewcommand{\scalingIk}[2]{\scalingsymbol_{#1}^{#2}}%
In our numerical experiments we use the basic congestion model~$\smoothcongestionj{j}{\OTxj{j}}= \frac{1}{2}\normx{\OTxj{j}}{\Convexi{j}}^2$ for~$j=1,\dots,\ny$, where~$\Convexi{j}\succeq 0$,
which implies convexity in the sense of Assumption~\ref{assumption:smoothnonsmooth}\eqref{assumption:smoothnonsmooth:nonsmooth} with characteristic~$\Convex=\diag({\Convexi{1},\dots,\Convexi{\ny}})$.
In particular, if we set~$\Convexi{j}=\strongconvexitymodulusj{j}\identityj{\primaldimensioni{j}}$ for some~$\strongconvexitymodulusj{j}\geq 0$, 
the primal step on
\cref{algorithm:primaldual:genericxIk}  of \cref{algorithm:newgenericprimaldual} reduces to 
\begin{equation}\label{scaledprojectedgradient}
\begin{array}{rcl}
\xIk{\altcoordinate}{\iter+1}
&
=
&\displaystyle
\Prox{\indicatorfunction{\OTxsetj{j}}}{
\big({\strongconvexitymodulusj{j}+\scalingIk{\altcoordinate}{\iter}}\big)^{-1}
\big({\scalingIk{\altcoordinate}{\iter}\xIk{\altcoordinate}{\iter}
-
\big(\costj{j}+\dualk{\iter}\big)}\big)
}
,
\end{array}
\end{equation}
where
we have used $\DimensionI{\altcoordinate}=({1}/{\probai{\altcoordinate}})\,\identityI{\primaldimensioni{\altcoordinate}}$ and either the constant preconditiong
$\StepsizekI{\iter}{\altcoordinate}=\StepsizeI{\altcoordinate}$ with
$
\StepsizeI{\altcoordinate} 
=
(1/{\stepsizei{\altcoordinate}} + \dualstepsize)\identityI{\primaldimensioni{\altcoordinate}}
$ derived from~\eqref{earlydiagonalscaling}, or the decreasing sequence
$\StepsizekI{\iter}{\altcoordinate}=\StepsizeI{\altcoordinate}/\stepsizek{\iter}$ with $\StepsizeI{\altcoordinate}
=
(\strongconvexitymodulusj{j}/\probai{\altcoordinate}^{2})\identityI{\primaldimensioni{\altcoordinate}}
$ and~$\stepsizek{\iter}$ satisfying~\eqref{stepsizesequence}, which respectively give
$
\scalingIk{\altcoordinate}{\iter}
=
({1}/{\probai{\altcoordinate}})
(1/{\stepsizei{\altcoordinate}} + \dualstepsize)
$
and
$
\scalingIk{\altcoordinate}{\iter}
=
(\strongconvexitymodulusj{j}\probai{\altcoordinate}/\stepsizek{\iter})
$.
The proximal operation in~\eqref{scaledprojectedgradient} thus involves a (scaled) gradient descent and an Euclidean projection on~$\OTxsetj{j}$, which are straightforward to compute.

\smallskip

\begin{table}
\begin{subtable}{1.0\linewidth}\centering
\begin{tabular}{|r||rrrr|rrr|}
\hline
$(\nx,\ny):$ 
& 
$(10,10)$ & $(20,20)$ & $(50,50)$ & $(100,100)$ & $(10,40)$ & $(10,250)$ & $(10,1000)$
\\
\hline
Constant ($\dualstepsize=1$):
&
261
&
519
&
1\,091
&
1\,927
&
914
&
3\,929
&
9\,350
\\
Constant (f.t. $\dualstepsize$):
&
221
&
99
&
44
&
56
&
41
&
96
&
278
\\
Accelerated:
&
130
&
128
&
152
&
122
&
107
&
92
&
62
\\
\hline
\end{tabular}
\caption{Epochs until $\normx{\Aall\xall-\bvector}{\infty}
 <
 10^{-6}$}\label{table:experiments:feas}
\bigskip
\end{subtable}
\begin{subtable}{1.0\linewidth}\centering
\begin{tabular}{|r||rrrr|rrr|}
\hline
$(\nx,\ny):$ 
& 
$(10,10)$ & $(20,20)$ & $(50,50)$ & $(100,100)$ & $(10,40)$ & $(10,250)$ & $(10,1000)$
\\
\hline
Constant ($\dualstepsize=1$):
&
409
&
816
&
1\,988
&
3\,861
&
1\,655
&
10\,762
&
/
\\
Constant (f.t. $\dualstepsize$):
&
221
&
99
&
53
&
64
&
45
&
141
&
472
\\
Accelerated:
&
1\,589
&
1\,288
&
2\,333
&
1\,094
&
1\,092
&
1\,771
&
1\,773
\\
\hline
\end{tabular}
\caption{Epochs until $ \bigdistt{\infty}{(\partial_{\xall},-\partial_{\dualvariable})\lagrangian{\xallk{\iter}}{\dualk{\iter}}}{0}
 <
 10^{-6}
$  where $\lagrangian{\xall}{\dualvariable}=\smooth{\xall} + \nonsmooth{\xall}+\inner{\dualvariable}{\Aall\xall-\bvector}$}\label{table:experiments:KKT}
\end{subtable}
\caption{Application of \cref{algorithm:newgenericprimaldual}  to~\eqref{problem:scheduling} with $\smoothcongestionj{j}{\OTxj{j}}= \frac{1}{2}\norm{\OTxj{j}}^2$. Constant~\eqref{earlydiagonalscaling} with~$\dualstepsize=1$ and with finely-tuned~$\dualstepsize$, and accelerated~\eqref{stepsizesequence} scaling. 
\label{table:experiments}}
\end{table}







In our tests, we set $\strongconvexitymodulusij{1}{j}=\cdots=\strongconvexitymodulusij{\nx}{j}=1
$ for~$j=1,\dots,\ny$. 
All the entries of $\costj{1},\dots,\costj{\dimension}$, $\xmeasurei{1}$,\dots,$\xmeasurei{\nx}$, and $\ymeasurei{1}$,\dots, $\ymeasurei{\ny}$ are picked randomly, independently and uniformly in $[0,1]$. The populations $\xmeasurei{1}$,\dots,$\xmeasurei{\nx}$ are then normalized so that the total load over total capacity saturation factor is~$0.8$ i.e., $\sum_{i=1}^{\nx}\xmeasurei{i}= 0.8 \sum_{j=1}^{\ny}\xmeasurei{j}$.
\cref{algorithm:newgenericprimaldual} is then implemented with the following block selection policy:
at every step, each agent updates their primal variables independently with probability~$1/\dimension$. Ignoring all update-free steps, which occur with probability $\probai{0}=(1-1/\dimension)^{\dimension}$, we find~$\probai{1}=\cdots=\probai{\dimension}=[\dimension (1-\probai{0})]^{-1}$ and $\probaij{\coordinate}{\altcoordinate}=[\dimension^2 (1-\probai{0})]^{-1}$ if $\coordinate\neq \altcoordinate$.
For the constant stepsize policy, we fix~$\dualstepsize$ and we set
$\stepsizei{1}=\cdots=\stepsizei{\dimension}
=
(1/2\dualstepsize)\,\dimension (1-\probai{0}) \,(\spectralradius{\constraintcovariance}-1) ^{-1}
$ in accordance with~\eqref{earlydiagonalscaling}. 
For the accelerated implementation, the decreasing stepsize sequence~$\stepsizek{\iter}$ is computed as in~\eqref{stepsizesequence} with $\stepsizek{0}=1$. Since in this problem $\NSmoothness=0$, we have $\betaparameter=\conditioning=0$. 

\cref{table:experiments} displays the number of updates per agent (epochs) that were needed to approach optimality. As stopping criteria we consider the feasiblity residual~$\normx{\Aall\xall-\bvector}{\infty}$ and the KKT residual $ \bigdistt{\infty}{(\partial_{\xall},-\partial_{\dualvariable})\lagrangian{\xallk{\iter}}{\dualk{\iter}}}{0}$, where $\lagrangian{\xall}{\dualvariable}=\smooth{\xall} + \nonsmooth{\xall}+\inner{\dualvariable}{\Aall\xall-\bvector}$ denotes the Lagrangian of~\eqref{initialP}.
In our simulations, the constant scaling implementation proved quite sensitive to the chosen parameters. 
Setting~$\dualstepsize=1$ was fine with smaller networks, though performance deteriorated as the size of the problem increased, and in particular for a larger number~$\ny$ of agents. Speed of convergence can be considerably increased by fine tuning the parameter values, as can be seen on the second lines of \cref{table:experiments:feas} and \cref{table:experiments:KKT}, where we used~$\dualstepsize=0.1$ for~$(\nx,\ny)=(10,10)$ and~$\dualstepsize=0.01$ for the larger networks.
The accelerated scaling, on the other hand, offered good and consistent (near dimension-free) performance even in the larger networks with tens of thousands of variables and many  agents.

\OBoff{$\smoothcongestionj{j}{\OTzj{j}}= v_{\text{out}}(\OTzj{j}/\ymeasurei{j})^t$}%
\OBoff{$\smoothcongestionj{j}{\OTzj{j}(\OTxj{j})}
= 
a \, (\OTzj{j}/\ymeasurei{j})^t
=
a \, (\sum_{i=1}^{\nx}\OTxij{i}{j}/\ymeasurei{j})^t
$\\
$\grad\smoothcongestionj{j}{\OTzj{j}(\OTxj{j})}
= 
a \, (\sum_{i=1}^{\nx}\OTxij{i}{j}/\ymeasurei{j})^t
$
}%
\OBoff{Congestion model 
\\
$
\smoothcongestionj{j}{\OTzj{j}}
=
a_{j}/(\ymeasurei{j}-\OTzj{j})
$
\\
$\smoothcongestionj{j}{\OTzj{j}(\OTxj{j})}
= 
a_{j}/(\ymeasurei{j}-\sum_{i=1}^{\nx}\OTxij{i}{j})
$
\\
$
\grad\smoothcongestionj{j}{\OTzj{j}(\OTxj{j})}
= 
\big(a_{j}/(\ymeasurei{j}-\sum_{i=1}^{\nx}\OTxij{i}{j})^2\big)_{i}
=
\big(a_{j}/(\ymeasurei{j}-\OTzj{j}(\OTxj{j}))^2\big)_{i}
$
\\
$
\Hessian\smoothcongestionj{j}{\OTzj{j}(\OTxj{j})}
= 
\big(2a_{j}/(\ymeasurei{j}-\sum_{i=1}^{\nx}\OTxij{i}{j})^3\big)_{i}
=
\big(2a_{j}/(\ymeasurei{j}-\OTzj{j}(\OTxj{j}))^3\big)_{i}
\geq
2a_{j}/\ymeasurei{j}
$ if $0\leq \OTzj{j} < \ymeasurei{j}$. Hesian has~$0$ determinant (directions $\Delta\OTz=0$).
}%


\appendix
%


\section{Proofs of Theorems~\ref{theorem:convergence} and~\ref{theorem:newscconvergence}}

\subsection{Convex case ($\Convex=0$)}

If $\Convex=0$, then we take the constant stepsizes~$\stepsizek{\iter}=1$ and~$\dualstepsizek{\iter}=\dualstepsize$ all~$\iter$, so that~\eqref{stepsizeconditionone} reduces to a condition on~$\dualstepsize$,
\begin{equation}\label{stepsizeconditionconvex}
   \dualstepsize\constraintcovariance+\Smoothness \preceq \Dimension\Stepsize.
\end{equation}
Besides,  in  Algorithm~\ref{algorithm:newgenericprimalonly} we get $\Sigmak{\iter}=(\iter+1)\dualstepsize$, while
$\penalizedk{\iter}= \hatcompositek{\iter}  -\compositeopt
 + \dualstepsize\iter (\constraint{\Sallk{\iter}}-\constraintmin ) $ and
 $\lyapunovk{\iter}{\xall} = \frac{\dualstepsize}{2}    \normx{\xallk{\iter}-\xall}{\Dimension^2\Stepsize}^2   +  \iter\dualstepsize  \penalizedk{\iter}$. 
%
%
%
%

We are now in a position to show Theorem~\ref{theorem:convergence}. Part~\eqref{theorem:convergence:i} in the proof, in particular, follows the lines of~\cite[proof of Theorem~1, part~(i)]{luke18}, with a few changes required by the presence of the smooth component~$\smoothfunction$. All the developments are detailed below for completeness.

\smallskip

\begin{proof}[Proof of Theorem~\ref{theorem:convergence}]

\eqref{theorem:convergence:i}
First suppose there is a Lagrange multiplier. Then,
\eqref{boundkpenalizedkSC} becomes
\begin{equation}\label{boundkpenalizedk}
 \iter\penalizedk{\iter} 
\geq
\frac{\dualstepsize\iter^2}{2} (\constraint{\Sallk{\iter}}-\constraintmin )   
-   \frac{\norm{\Aall\multipliersol}^2}{\dualstepsize}
\geq 
-   \frac{\norm{\Aall\multipliersol}^2}{\dualstepsize} 
,
\end{equation}
and we have $\lyapunovk{\iter}{\xall}\geq -\norm{\Aall\multipliersol}^2$. The supermartingale theorem applies, and~\eqref{doob} reduces to
\begin{equation}
 \sum_{\iter=0}^{\infty} 
 \left[
    \frac{1}{2} \normx{\hatxallk{\iter+1}-\xallk{\iter}}{\Dimension\Stepsize-(\dualstepsize\constraintcovariance+\NSmoothness)}^2  
+
\dualstepsize (\constraint{\xallk{\iter}}-\constraintmin)
 \right]
 <
 \infty
 \quad
 \text{a.s.}
 ,
\end{equation}
which almost surely yields  $\constraint{\xallk{\iter}}-\constraintmin=\zero{1/\iter}$ and $\normx{\hatxallk{\iter+1}-\xallk{\iter}}{\Dimension\Stepsize-(\dualstepsize\constraintcovariance+\NSmoothness)}^2\to 0$ (hence, by~\eqref{constraintidentity}, $\constraint{\hatxallk{\iter+1}}-\constraintmin=\frac{1}{2} \normx{\hatxallk{\iter+1}-\xallsol}{\Aall^\transpose\Aall}^2 \to 0$), and the sequence~$\lyapunovk{\iter}{\xallsol} $ converges almost surely to a random quantity not smaller than  $-\norm{\Aall\multipliersol}^2$.  It follows that~$\lyapunovk{\iter}{\xallsol} $ is almost surely bounded from above by a constant~$\constantC(\sample)$, so that
\begin{equation}
 \constantC(\sample) + \norm{\Aall\multipliersol}^2
{\,\geq\,}
\lyapunovk{\iter}{\xallsol}  + \norm{\Aall\multipliersol}^2
 {\,=\,} 
 \frac{\dualstepsize}{2}    \normx{\xallk{\iter}-\xallsol}{\Dimension^2\Stepsize}^2  +  \dualstepsize \iter  \penalizedk{\iter}  + \norm{\Aall\multipliersol}^2
{\,\refereq{\eqref{boundkpenalizedk}}{\geq}\,} 
 \frac{\dualstepsize}{2}    \normx{\xallk{\iter}-\xallsol}{\Dimension^2\Stepsize}^2
 {\,\geq\,}
0
\ \text{a.s.}
,
\end{equation}
which shows that the sequence~$\xallk{\iter}$ is pointwise almost surely bounded, and so are the sequences~$\Sallk{\iter}$ and~$\Zallk{\iter}$ as a consequence of~\eqref{ergodicaverage}. Since almost surely 
$
\iter\dualstepsize  \penalizedk{\iter}
=
\lyapunovk{\iter}{\xallsol}  
 -
\frac{\dualstepsize}{2}    \normx{\xallk{\iter}-\xallsol}{\Dimension^2\Stepsize}^2 
\leq
\lyapunovk{\iter}{\xallsol}
$
is bounded from above, we infer from~\eqref{boundkpenalizedk} that $\constraint{\Sallk{\iter}}-\constraintmin=\magnitude{1/\iter^2}$ almost surely.
Since we also have $\constraint{\xallk{\iter}}-\constraintmin=\zero{1/\iter}$ almost surely, $\Sigmak{\iter-1}/{\Sigmak{\iter}}=\magnitude{1}$, and $\dualstepsizek{\iter}/\Sigmak{\iter}=\magnitude{1/\iter}$ as $\iter\to\infty$,
it follows from   
Line~\ref{algorithm:genericZallk} of Algorithm~\ref{algorithm:newgenericprimalonly} and from the convexity of~$\constraintfunction$ that $\constraint{\Zallk{\iter}}-\constraintmin=\magnitude{1/\iter^2}$ almost surely. Then, using $\Aall\xallsol=\bvector$, we find
\begin{equation}\label{boundinnerproduct}
\begin{array}{l}
 \inner{\grad \constraint{\Zallk{\iter}}}{\xallsol-\hatxallk{\iter+1}} 
 \refereq{\eqref{constraint}}{=}
 \inner{\Aall^\transpose(\Aall\Zallk{\iter}-\bvector)}{\xallsol-\hatxallk{\iter+1}} 
 \qquad\\\hfill
 =
 \inner{\Aall^\transpose\Aall(\Zallk{\iter}-\xallsol)}{\xallsol-\hatxallk{\iter+1}} 
 \leq
\norm{\Aall(\Zallk{\iter}-\xallsol)}\norm{\Aall(\xallsol-\hatxallk{\iter+1})} 
 \qquad\\\hfill
\refereq{\eqref{constraintidentity}}{=}
2 \sqrt{\constraint{\Zallk{\iter}}-\constraintmin}\,\sqrt{\constraint{\hatxallk{\iter+1}}-\constraintmin}
= \zero{1/\iter}
,
 \end{array}
\end{equation}
where we have used $\constraint{\hatxallk{\iter+1}}-\constraintmin \to 0$. By convexity and smoothness of~$\smoothfunction$, we find
\begin{equation}\label{convexlipschitzresult}
\begin{array}{rcl}
\inner{\grad\smooth{\xallk{\iter}}}{\xallsol-\hatxallk{\iter+1}}
&=&
\inner{\grad\smooth{\xallk{\iter}}}{\xallsol-\xallk{\iter}}
-
\inner{\grad\smooth{\xallk{\iter}}}{\hatxallk{\iter+1}-\xallk{\iter}}
\\
&\refereq{\eqref{smoothnessconvexity}}{\leq}&
[\smooth{\xallsol}-\smooth{\xallk{\iter}}]
-
[\smooth{\hatxallk{\iter+1}}-\smooth{\xallk{\iter}} - \frac{1}{2}\normx{\hatxallk{\iter+1}-\xallk{\iter}}{\Smoothness}^2]
\\
&=&
\smooth{\xallsol} - \smooth{\hatxallk{\iter+1}} + \frac{1}{2}\normx{\hatxallk{\iter+1}-\xallk{\iter}}{\Smoothness}^2
.
\end{array}
\end{equation}

Consider a bounded sequence~$(\xallk{\iter})$ (this event happens with probability one), and one of its convergent subsequence~$(\xallk{\iter_\altiter})$ with accumulation point~$\altxall$, which is feasible since $\constraint{\altxall}=\lim_{\altiter\to\infty}\constraint{\xallk{\iter_\altiter}}=\constraintmin$.
Now, applying  Lemma~\ref{lemma:descentargumentcostnonsmooth} at point~$\xallsol$ gives
\begin{equation}\notag
\begin{array}{rcl}
\Nnonsmooth{\hatxallk{\iter_\altiter+1}} - \Nnonsmooth{\xallsol}
&\refereq{\eqref{descentargument}}{\leq}&
\inner{\grad\smooth{\xallk{\iter_\altiter}} + \Sigmak{\iter_\altiter}\grad \constraint{\Zallk{\iter_\altiter}} + \Dimension\Stepsizek{\iter_\altiter}(\hatxallk{\iter_\altiter+1}-\xallk{\iter_\altiter})  }{\xallsol-\hatxallk{\iter_\altiter+1}} 
 \\ 
&\refereq{\eqref{convexlipschitzresult}}{\leq}&
\smooth{\xallsol} - \smooth{\hatxallk{\iter_\altiter+1}} + \frac{1}{2}\normx{\hatxallk{\iter_\altiter+1}-\xallk{\iter_\altiter}}{\Smoothness}^2 
\qquad\qquad\qquad \\ &&\hfill
+ \inner{\Sigmak{\iter_\altiter}\grad \constraint{\Zallk{\iter_\altiter}} + \Dimension\Stepsizek{\iter_\altiter}(\hatxallk{\iter_\altiter+1}-\xallk{\iter_\altiter})  }{\xallsol-\hatxallk{\iter_\altiter+1}} 
,
\end{array}
\end{equation}
which rewrites as
\begin{equation}
\begin{array}{l}
 \composite{\hatxallk{\iter_\altiter+1}} - \composite{\xallsol}
\leq
\frac{1}{2}\normx{\hatxallk{\iter_\altiter+1}-\xallk{\iter_\altiter}}{\Smoothness}^2 
+ \Sigmak{\iter_\altiter} \inner{\grad \constraint{\Zallk{\iter_\altiter}}}{\xallsol-\hatxallk{\iter_\altiter+1}} 
\qquad\\\hfill
+ \inner{\Dimension\Stepsizek{\iter_\altiter}(\hatxallk{\iter_\altiter+1}-\xallk{\iter_\altiter})  }{\xallsol-\hatxallk{\iter_\altiter+1}} 
.
\end{array}
\end{equation}
Since $\normx{\hatxallk{\iter_\altiter+1}-\xallk{\iter_\altiter}}{\Dimension\Stepsize-(\dualstepsize\constraintcovariance+\NSmoothness)}^2\to 0$, the sequence $(\hatxallk{\iter_\altiter+1})$ is bounded and, by taking the limit of the last equation for $\altiter\to\infty$, we find, using~\eqref{boundinnerproduct}, $\Sigmak{\iter_\altiter}=\iter_\altiter\dualstepsize$ and the lower-semicontinuity of~$\compositefunction$: $\composite{\altxall}  \leq \lim_{\altiter\to\infty}  \composite{\hatxallk{\iter_\altiter+1}} \leq \composite{\xallsol} $, and $\xallsol\in\mathcal{S}$.
To show that the bounded sequence~$(\xallk{\iter})$ has a unique accumulation point, we consider a second convergent subsequence $\xallk{\bar\iter_\altiter}\to\altaltxall$.
Recall that $\lyapunovk{\iter}{\xall} =  \frac{\dualstepsize}{2}    \normx{\xallk{\iter}-\xall}{\Dimension^2\Stepsize}^2   +  \iter\dualstepsize  \penalizedk{\iter}$ and that $\lyapunovk{\iter}{\xall}$ is convergent for our bounded sequence~$(\xallk{\iter})$, thus admitting the same limit for subsequences~$(\xallk{\iter_\altiter})$ and~$(\xallk{\bar\iter_\altiter})$. Computing these limits at $\xall=\altxall$ yields $\lim_{\altiter\to\infty}   \lyapunovk{\iter_\altiter}{\altxall} =  \lim_{\altiter\to\infty}  \iter_\altiter\dualstepsize  \penalizedk{\iter_\altiter}$ and $\lim_{\altiter\to\infty}   \lyapunovk{\bar\iter_\altiter}{\altxall} =  \frac{\dualstepsize}{2}    \normx{\altaltxall-\altxall}{\Dimension^2\Stepsize}^2   + \lim_{\altiter\to\infty}  \bar\iter_\altiter\dualstepsize  \penalizedk{\bar\iter_\altiter}$, hence $\lim_{\altiter\to\infty}  \iter_\altiter\dualstepsize  \penalizedk{\iter_\altiter} = \frac{\dualstepsize}{2}    \normx{\altaltxall-\altxall}{\Dimension^2\Stepsize}^2   + \lim_{\altiter\to\infty}  \bar\iter_\altiter\dualstepsize  \penalizedk{\bar\iter_\altiter}$. Repeating this operation at $\xall=\altaltxall$ gives $\frac{\dualstepsize}{2}    \normx{\altxall-\altaltxall}{\Dimension^2\Stepsize}^2 + \lim_{\altiter\to\infty}  \iter_\altiter\dualstepsize  \penalizedk{\iter_\altiter} =     \lim_{\altiter\to\infty}  \bar\iter_\altiter\dualstepsize  \penalizedk{\bar\iter_\altiter}$. Thus $\normx{\altxall-\altaltxall}{\Dimension^2\Stepsize}^2=0$ and $\altxall=\altaltxall$, which shows that almost surely~$(\xallk{\iter})$ converges pointwise to a point of~$\mathcal{S}$.
Let~$\xallsol$ be that limit point.
From \cref{lemma:extrapolation}, we know that 
$
\sum_{\altiter=0}^{\iter} |\gammacoefkl{\iter}{\altiter}|
=
\sum_{\altiter=0}^{\iter} \gammacoefkl{\iter}{\altiter}
=
\identityI{\primaldimension}
$.
Moreover, \eqref{gammacoefkl} gives
$ \gammacoefkl{\iter+1}{\altiter}=(\Sigmak{\altiter}/\Sigmak{\iter}) \gammacoefkl{\altiter+1}{\altiter}=\magnitude{1/\iter}$ and thus, for all~$\altiter$, $\gammacoefkl{\iter+1}{\altiter}\to 0$ as $\iter\to\infty$. It follows from~\eqref{Sallk} and the Toeplitz Theorem, \cite{Toeplitz1911}, that 
$
 \Sallk{\iter} = \sum_{\altiter=0}^{\iter} \gammacoefkl{\iter}{\altiter} \xallk{\altiter}
 \to \xallsol \in\mathcal{S}
$
almost surely, which completes the proof of~\eqref{theorem:convergence:i}.

%

\smallskip
\eqref{theorem:convergence:ii}
%
%
We know from Section~\ref{section:withoutmultiplier} that almost surely $\scaledlyapunovk{\iter}{\xallsol}$ converges  to $Z(\sample)\geq \compositemin-\compositeopt$ and is thus bounded from above by a constant~$\constantC(\sample)$, so that $\constantC(\sample)\geq \scaledlyapunovk{\iter}{\xallsol}
\geq  \penalizedk{\iter}$. It follows from~\eqref{boundkpenalizedkSCwithout} that
%
%
$
\constraint{\Sallk{\iter}}-\constraintmin
\leq 
[  
\penalizedk{\iter} 
+\compositeopt - \compositemin]
/(\iter\dualstepsize)
\leq 
[  
\constantC(\sample) 
+\compositeopt - \compositemin]
/(\iter\dualstepsize)
$ and, consequently, we find $\constraint{\Sallk{\iter}}-\constraintmin
=\magnitude{1/\iter}$ almost surely.
From~\eqref{doob} we also have $(\constraint{\xallk{\iter}}-\constraintmin)=\zero{1} $ almost surely.
It follows from \cref{algorithm:genericZallk} in \cref{algorithm:newgenericprimalonly} and from the convexity of~$\constraintfunction$ that almost surely $\constraint{\Zallk{\iter}}-\constraintmin
=\magnitude{1/\iter}$.
%
%

Now, consider the sublevel sets defined by
\begin{equation}\label{eqdiese}
 \sublevelset{c} = \{\xall : \max(\constraint{\xall}-\constraintmin , \composite{\xall} - \compositeopt )\leq c  \},
 \qquad c\geq 0
 .
\end{equation}
Since $\sublevelset{0} = \mathcal{S}$ is bounded and both~$\constraintfunction$ and~$\compositefunction$ are convex functions, the sets~$\sublevelset{c}$ are bounded for all~$c\geq 0$.
From $\hatcompositek{\iter}\geq\composite{\Sallk{\iter}}$ and the definitions of~$\scaledlyapunovk{\iter}{\xallsol}$, we find
$
\scaledlyapunovk{\iter}{\xall} 
\geq
  \composite{\Sallk{\iter}} -\compositeopt
 + \Sigmak{\iter-1} (\constraint{\Sallk{\iter}}-\constraintmin )
$, which yields
\begin{equation} \label{almostsureanalysis}
\constraint{\Sallk{\iter}}-\constraintmin 
\leq
\frac{1}{\Sigmak{\iter-1}} (\scaledlyapunovk{\iter}{\xallsol} +  \compositeopt -\compositemin)
\quad \text{and} \quad
\composite{\Sallk{\iter}} - \compositeopt
\leq
\scaledlyapunovk{\iter}{\xallsol}  
,
\end{equation}
in which $\scaledlyapunovk{\iter}{\xallsol}\to Z(\sample)$ for any outcome $\sample\in\Omega$, where~$\Omega$ is an event of probability~$1$. It follows that both sequences $(\constraint{\Sallk{\iter}(\sample)}-\constraintmin)$ and
$(\composite{\Sallk{\iter}(\sample)} - \compositeopt)$ are bounded by a constant~$\constantCprime(\sample)$. Hence, the sequence~$\Sallk{\iter}(\sample)$ belongs to~$\sublevelset{\constantCprime(\sample)}$, which is a bounded set. Thus, the sequence~$\Sallk{\iter}(\sample)$ is bounded if $\sample\in\Omega$.
%

Since $\Sigmak{\iter-1}=\dualstepsize\iter$, \eqref{almostsureanalysis} also yields $\constraint{\Sallk{\iter}(\sample)}-\constraintmin = \magnitude{1/\iter}$ if $\sample\in\Omega$. 
Hence $\constraint{\Sallk{\iter}(\sample)}-\constraintmin = \magnitude{1/\iter}$ almost surely, as seen previously. 
Besides, all the accumulation points of~$(\Sallk{\iter}(\sample))$ are feasible if $\sample\in\Omega$. It follows that $\lim\inf_{\iter\to\infty} \composite{\Sallk{\iter}} \geq \compositeopt$ if $\sample\in\Omega$, hence with probability one.
Now, observe that~\eqref{fortyeight} reduces in the convex case to
\begin{equation} \label{fortyeightconvex}
 \Expectation{ \frac{1}{2\iter}  \normx{\xallk{\iter}-\xallsol}{\Dimension^2\Stepsize+\Dimension\Convex}^2  } +   \dualstepsize\iter  \expectation{ \constraint{\Sallk{\iter}}-\constraintmin }   + \expectation{
    \composite{\Sallk{\iter}} -\compositeopt }
\leq  
\frac{\constantC}{\dualstepsize\iter} 
,
\end{equation} 
%
%
%
and suppose that there is an event $\tilde\Omega$ with nonzero probability such that $\lim\sup_{\iter\to\infty} \composite{\Sallk{\iter}(\sample)} > \compositeopt$. Then, taking the limit superior in~\eqref{fortyeightconvex} leads to a contradiction. Hence, $\lim\sup_{\iter\to\infty} \composite{\Sallk{\iter}} \leq \compositeopt$ with probability one. 
%
Thus,  $\lim_{\iter\to\infty} \composite{\Sallk{\iter}} = \compositeopt$ with probability one, and almost surely the accumulation points of~$(\Sallk{\iter})$ belong to~$\mathcal{S}$.

It remains to derive the convergence rate of the sequence~$(\constraint{\Sallk{\iter}(\sample)})$. We know that  $ \composite{\Sallk{\iter}(\sample)} - \compositeopt  = \zero{1}$ for every $\sample\in\Omega$.
Consider the nonnegative sequence $u^{\iter}(\sample)=\max(0,\compositeopt-\composite{\Sallk{\iter}(\sample)} )$.
For all~$\iter$ and for all~$\sample\in\Omega$, we have
$ 
\condexpectation{\sample\in\Omega}{\compositeopt-  \composite{\Sallk{\iter}(\sample)}}
\leq
\condexpectation{\sample\in\Omega}{u^{\iter}(\sample)}
$
where
$ (u^{\iter}) $ converges pointwise towards~$0$ on~$\Omega$.
Since~$(u^{\iter})$
satisfies
$|u^{\iter}(\sample)|\leq \compositeopt-\compositemin$ for all~$\iter$,
with 
$
\condexpectation{\sample\in\Omega}{|\compositeopt-\compositemin|}=\compositeopt-\compositemin<\infty
$, 
Lebesgue's dominated convergence theorem applies, and we find $\lim_{\iter\to\infty}\condexpectation{\sample\in\Omega}{u^{\iter}(\sample)}=0$.
It follows that
$
\lim_{\iter\to\infty}\expectation{\compositeopt-  \composite{\Sallk{\iter}(\sample)}} \leq \proba{\sample\in\Omega} \, \lim_{\iter\to\infty}\condexpectation{\sample{\,\in\,}\Omega}{u^{\iter}(\sample)} + (1-\proba{\sample\in\Omega}) \,\lim_{\iter\to\infty}\condexpectation{\sample{\,\notin\,}\Omega}{u^{\iter}(\sample)} \leq 0
$, 
where we have used $\proba{\sample\in\Omega}=1$ and $u^{\iter}\leq \compositeopt-\compositemin $.
Combining this result with~\eqref{fortyeightconvex} gives
\begin{equation*} 
     \expectation{ \constraint{\Sallk{\iter}}-\constraintmin }    
\leq  
\frac{\constantC}{(\dualstepsize\iter)^2} 
+
\frac{\expectation{
   \compositeopt - \composite{\Sallk{\iter}}  }}{\dualstepsize\iter}
=
\zero{{1}/{\iter}}
,
\end{equation*} 
which completes the proof of~\cref{theorem:convergence}.
%
%
%
%
%
%
%
%
\end{proof}


\subsection{Strongly convex case}

The proof of~\ref{theorem:newscconvergence} under strong convexity is similar to that of Theorem~\ref{theorem:convergence}. Recall from Section~\ref{section:newstronglyconvexanalysis}, however, that in Algorithm~\ref{algorithm:newgenericprimalonly} we now have $\stepsizek{\iter}-2/\iter=\zero{1/\iter}$,  $\dualstepsizek{\iter} - ({\alpha\iter}/{2} - \fixed{\betaparameter} ) = \zero{1} $,  and $ \Sigmak{\iter}- [ \alpha\iter(\iter+1)/{4} - (\iter+1) \fixed{\betaparameter}] = \zero{\iter} $.

\smallskip
\begin{proof}[Proof of Theorem~\ref{theorem:newscconvergence}]

\eqref{theorem:newscconvergence:i} 
We begin as in the proof of Theorem~\ref{theorem:newscconvergence}\eqref{theorem:newscconvergence:i} by
showing that~$\iter\penalizedk{\iter}$ is bounded from below.
If~\eqref{genericP}  admits an optimal primal-dual pair $(\xallsol,\multipliersol)$ satisfying~\eqref{boundcompositeSallk},
then, 
it follows from~\eqref{boundkpenalizedkSC} that the sequence $\lyapunovk{\iter}{\xall}\geq -\norm{\Aall\multipliersol}^2$ is bounded from below, and the supermartingale theorem applies in~\eqref{earlystronglyconvexfejersimple} to the nonnegative sequence $\lyapunovk{\iter}{\xall} + \norm{\Aall\multipliersol}^2$, so that
\begin{equation}
 \sum_{\iter=0}^{\infty} 
 \left[
    \frac{\dualstepsizek{\iter}}{2\stepsizek{\iter}} \normx{\hatxallk{\iter+1}-\xallk{\iter}}{\Dimension\Stepsize-\stepsizek{\iter}(\dualstepsizek{\iter}\constraintcovariance+\NSmoothness)}^2    
+
(\dualstepsizek{\iter})^2 (\constraint{\xallk{\iter}}-\constraintmin)
 \right]
 <
 \infty
 \quad
 \text{a.s.}
 ,
\end{equation}
where $\Dimension\Stepsize-\stepsizek{\iter}(\dualstepsizek{\iter}\constraintcovariance+\NSmoothness)\to \Dimension\Stepsize-\alpha\constraintcovariance$,
which almost surely yields  $\constraint{\xallk{\iter}}-\constraintmin=\zero{1/\iter^3}$ and $\iter^2\normx{\hatxallk{\iter+1}-\xallk{\iter}}{\Dimension\Stepsize}^2\to 0$ (hence, by~\eqref{constraintidentity}, $\iter^2[\constraint{\hatxallk{\iter+1}}-\constraintmin ] =  \frac{\iter^2 }{2} \normx{\hatxallk{\iter+1}-\xallsol}{\Aall^\transpose\Aall}^2 \to 0$), and the sequence~$\lyapunovk{\iter}{\xallsol} $ converges almost surely to a random quantity not smaller than  $-\norm{\Aall\multipliersol}^2$.  It follows that~$\lyapunovk{\iter}{\xallsol} $ is almost surely bounded from above by a constant~$\constantC$, so that
\begin{equation}
 \constantC + \norm{\Aall\multipliersol}^2
\geq 
\lyapunovk{\iter}{\xallsol}  + \norm{\Aall\multipliersol}^2
\refereq{\eqref{boundkpenalizedkSC}}{\geq} 
 \frac{\dualstepsizek{\iter}}{2\stepsizek{\iter}}    \normx{\xallk{\iter}-\xallsol}{\Dimension^2\Stepsize}^2    \geq
0
\quad \text{a.s.}
,
\end{equation}
which shows that $   \normx{\xallk{\iter}-\xallsol}{\Dimension^2\Stepsize}   =\magnitude{1/\iter}$ and, since~$\Dimension^2\Stepsize$ has full rank,  the sequence~$(\xallk{\iter})$ converges pointwise almost surely to the solution~$\xallsol$. Since 
\eqref{gammacoefkl} in \cref{lemma:extrapolation} gives
$ \gammacoefkl{\iter+1}{\altiter}=(\Sigmak{\altiter}/\Sigmak{\iter}) \gammacoefkl{\altiter+1}{\altiter}=\magnitude{1/\iter^2}$ for all~$\altiter$, 
%
%
so do sequences~$(\Sallk{\iter})$ and~$(\Zallk{\iter})$ as a consequence of~%
\eqref{Sallk}, \cref{algorithm:genericZallk} in \cref{algorithm:newgenericprimalonly}, and the Toeplitz Theorem. 
Since almost surely 
$ 
\Sigmak{\iter-1}  \penalizedk{\iter}
=
\lyapunovk{\iter}{\xallsol}  
 -
\frac{\dualstepsizek{\iter}}{2\stepsizek{\iter}}    \normx{\xallk{\iter}-\xallsol}{\Dimension^2\Stepsize+\stepsizek{\iter}\Dimension\Convex}^2
\leq
\lyapunovk{\iter}{\xallsol}
$
is bounded from above, we infer from~\eqref{boundkpenalizedkSC} that $\constraint{\Sallk{\iter}}-\constraintmin=\magnitude{1/\iter^4}$ almost surely.
Since we also have $\constraint{\xallk{\iter}}-\constraintmin=\zero{1/\iter^3}$ almost surely, $\Sigmak{\iter-1}/{\Sigmak{\iter}}=\magnitude{1}$, and $\dualstepsizek{\iter}/\Sigmak{\iter}=\magnitude{1/\iter}$ as $\iter\to\infty$,
it follows from   
Line~\ref{algorithm:genericZallk} of Algorithm~\ref{algorithm:newgenericprimalonly} and from the convexity of~$\constraintfunction$ that $\constraint{\Zallk{\iter}}-\constraintmin=\magnitude{1/\iter^4}$ almost surely.

\eqref{theorem:newscconvergence:ii}
In the strongly convex case, we know from Section~\ref{section:withoutmultiplier} that almost surely $\scaledlyapunovk{\iter}{\xallsol}$ converges  to $Z(\sample)\geq \compositemin-\compositeopt$ and is thus bounded from above by a constant~$\constantC(\sample)$, so that $\constantC(\sample)\geq \scaledlyapunovk{\iter}{\xallsol}
\geq  \penalizedk{\iter}$. It follows from~\eqref{boundkpenalizedkSCwithout} that
%
%
$
\constraint{\Sallk{\iter}}-\constraintmin
\leq 
[  
\penalizedk{\iter} 
+\compositeopt - \compositemin]
/\Sigmak{\iter-1}
\leq 
[  
\constantC(\sample) 
+\compositeopt - \compositemin]
/\Sigmak{\iter-1}
$ and, since $ \Sigmak{\iter} = \magnitude{\iter^2} $, we find $\constraint{\Sallk{\iter}}-\constraintmin
=\magnitude{1/\iter^2}$ almost surely.
From~\eqref{doob} we also have $(\constraint{\xallk{\iter}}-\constraintmin)=\zero{1/\iter} $ almost surely.
It follows from \cref{algorithm:genericZallk} in \cref{algorithm:newgenericprimalonly} and from the convexity of~$\smoothfunction$ that $\constraint{\Zallk{\iter}}-\constraintmin
=\magnitude{1/\iter^2}$ almost surely.

Since \eqref{almostsureanalysis} still holds, we find that the sequence~$\Sallk{\iter}(\sample)$ is bounded for $\sample\in\Omega$, 
with $\constraint{\Sallk{\iter}(\sample)}-\constraintmin = \magnitude{1/\iter^2}$. 

%
All the accumulation points of~$(\Sallk{\iter}(\sample))$ are thus feasible if $\sample\in\Omega$. As a consequence, we find $\lim\inf_{\iter\to\infty} \composite{\Sallk{\iter}} \geq \compositeopt$ almost surely.
Under strong convexity, \eqref{fortyeight} now reduces to
\begin{equation} \label{fortyeightsconvex}
 \Expectation{ \frac{1}{2\Sigmak{\iter-1}}  \normx{\xallk{\iter}-\xallsol}{\scalematrixk{\iter}}^2  } +   \Sigmak{\iter-1} \expectation{ \constraint{\Sallk{\iter}}-\constraintmin }  +    \expectation{ \composite{\Sallk{\iter}} -\compositeopt} 
\leq  
\frac{\constantC}{\Sigmak{\iter-1}} 
.
\end{equation} 
Reasoning as we did for~\eqref{fortyeightconvex}, we also find $\lim\sup_{\iter\to\infty} \composite{\Sallk{\iter}} \leq \compositeopt$ with probability one. 
It follow that, almost surely, $\lim_{\iter\to\infty} \composite{\Sallk{\iter}} = \compositeopt$ and~$\xallsol$ is the unique accumulation point of~$(\Sallk{\iter})$.
As for the convergence rate of the sequence~$(\constraint{\Sallk{\iter}(\sample)})$, it still holds that $ \composite{\Sallk{\iter}(\sample)} - \compositeopt  = \zero{1}$ almost surely, and that  $
\lim_{\iter\to\infty}\expectation{\compositeopt-  \composite{\Sallk{\iter}(\sample)}} \leq 0
$.
Applying this last resul to~\eqref{fortyeightconvex} yields
\begin{equation*} 
     \expectation{ \constraint{\Sallk{\iter}}-\constraintmin }    
\leq  
\frac{\constantC}{(\Sigmak{\iter-1})^2} 
+
\frac{\expectation{
   \compositeopt - \composite{\Sallk{\iter}}  }}{\Sigmak{\iter-1}}
=
\zero{{1}/{\iter^2}}
,
\end{equation*} 
which completes the proof of~\cref{theorem:newscconvergence}.
\end{proof}
\bibliographystyle{model2-names}








\bibliography{obPDLO-LA}

\end{document}